\documentclass{article}
\usepackage[utf8]{inputenc}
\usepackage{fullpage}
\usepackage{amsthm,mathtools,scalerel}
\usepackage{color,soul,latexsym,amsmath,amssymb,amsfonts,amsthm,dsfont,enumitem,xcolor,bbm,threeparttable,graphicx,float,subfigure}
\usepackage{authblk}
\usepackage[round]{natbib}
\RequirePackage[colorlinks=true,allcolors=blue,hypertexnames=false]{hyperref}%
\usepackage{graphicx,todonotes}
\usepackage{outlines}
\usepackage{array}
\usepackage{selectp}
\usepackage{forest}
\usepackage{enumitem}
\setlist[description]{style=nextline}
\usepackage[ruled,linesnumbered, vlined, noend]{algorithm2e}
\mathtoolsset{showonlyrefs}
\usepackage{titlesec}

\newtheorem{thm}{Theorem}%[section]

\newtheorem{lemma}{Lemma}

\newcommand{\E}{\mathbb{E}}

\newcommand{\Normal}{{\mathcal{N}}}

\renewcommand{\d}{\mathrm{d}}

\newcommand{\e}{\varepsilon}

 \newcommand{\bb}[1]{\mathbb{#1}}
\newcommand{\N}{\mathbb{N}}
\newcommand{\Z}{\mathbb{Z}}
\newcommand{\paren}[1]{\left(#1\right)}
\newcommand{\sqbrace}[1]{\left[#1\right]}
\newcommand{\cbrace}[1]{\left\{#1\right\}}
\newcommand{\norm}[2]{\left\|#1\right\|_{#2}}
\renewcommand{\l}{\lambda}

\newcommand{\bp}{\begin{proof}}
\newcommand{\ep}{\end{proof}}
\renewcommand{\L}{\Lambda}
\newcommand{\uind}[3]{1\leq #1 < #2 \leq #3}

\title{\bf On the L{\'e}vy concentration function of Gaussian quadratic forms with applications to second order U-statistics}
\author[1]{Abhimanyu Choudhary}
\author[2]{Arun Kumar Kuchibhotla}
\affil[1]{Department of Mathematical Sciences, Carnegie Mellon University}
\affil[2]{Department of Statistics and Data Science, Carnegie Mellon University}
\begin{document}
\maketitle
\begin{abstract}
    We provide an upper-bound for the L{\'e}vy concentration function:
    $$
    Q_{S}(\e):= \sup_{x \in\bb R}\bb P (x < S \leq x+\e)
    $$
    where $S$ is a weighted sum of noncentral chi-square random variables:
    $$
    S:= \sum_{k=1}^\infty \lambda_k (Z_k^2 - 1) + \mu_kZ_k
    $$
    Here, $Z_1, \cdots, Z_n$ are independent standard Gaussian random variables and $\{\lambda_k\}_{k=1}^\infty, \{\mu_k\}_{k=1}^\infty$ are real valued, square summable sequences. Random variables of this type often appear as limiting distributions of second order $U$-statistics. Our bound is adaptive, in that it recovers (up to constant factors) Gaussian type concentration function estimates if $\|\lambda\|_2$ is negligible compared to $\norm{\mu}{2}$ and chi-square estimates if $\norm{\mu}{2}$ is negligible compared to $\norm{\lambda}{2}$. Our bound generalizes existing bounds in various ways. In particular, we make no assumptions regarding the number of nonzero $|\lambda_k|$ or the size of the minimal $|\l_k|$, nor do we make any assumptions on the signs of $\lambda_k$. Finally, we apply our bound to some examples of interest, specifically quadratic forms that arise in limit theorems for second-order U-statistics.
\end{abstract}

\section{Introduction}
Given independent and identically distributed random variables $X_1, \ldots, X_n$ taking values in some measurable space $(\mathcal{X}, \mathcal{A})$ and a  measurable function $h:\mathcal{X}\times\mathcal{X}\to\mathbb{R}$, consider the second order $U$-statistic
\[
U_n = \frac{1}{{n \choose 2}}\sum_{1\le i \neq j \le n}\, h(X_i, X_j).
\]
This is an unbiased estimator of $\theta = \mathbb{E}[U_n] = \mathbb{E}[h(X_1, X_2)].$ In fact, $U_n$ is the UMVUE (uniformly minimum variance unbiased estimator) for $\theta$. $U$-Statistics are a natural first extension of sums of independent random variables. In the context of hypothesis testing, $U$-statistics appear as test statistics.  The sample variance is one of the most commonly used second order $U$-statistics (in the context of inference for the population mean). Additionally, the Hoeffding decomposition (a form of ANOVA decomposition) implies that any function of independent random variables can be approximated by a $U$-statistics of sufficiently high order~\citep{alberinkthes}. For these reasons, understanding the limiting distributions of $U$-statistics is a fundamental problem. The limit theory of $U$-statistics is much richer compared to that of IID sums. The limiting distribution can be a Gaussian or a Gaussian chaos (related to the chi-square distribution) depending the degeneracy of $h(\cdot, \cdot)$. The {\em kernel} $h(\cdot, \cdot)$ is called degenerate or canonical if
\[
\mathbb{E}[h(X_1, X_2)|X_2] = \mathbb{E}[h(X_1, X_2)|X_1] = C\quad\mbox{almost surely.}
\]
From now on, we assume that the kernels of the $U$-statistics are symmetric, i.e., $h(x_1, x_2) = h(x_2, x_1)$ and we set $H_1(x) = \mathbb{E}[h(x, X_1)]$. With this notation, the kernel is degenerate if $H_1(X) = C$ almost surely.
\newline 
\newline
To illustrate the difference in the limiting distributions, consider the example of estimating $\mu^2 = (\mathbb{E}[X_1])^2$ using
\[
\widehat{\theta}_n = \frac{1}{\binom{n}{2}}\sum_{i < j} X_iX_j ~=~ \frac{n}{n-1}\left(\frac{1}{n}\sum_{i=1}^n X_i\right)^2 - \frac{1}{n-1}\left(\frac{1}{n}\sum_{i=1}^n X_i^2\right).
\]
If $\mu = \mathbb{E}[X_1] = 0$ (degenerate case), then both terms are of order $1/n$ and if $\mu = \mathbb{E}[X_1] \neq 0$ (non-degenerate case), then the first term is of order $1/\sqrt{n}$ and the second term is of order $1/n$. This implies that the rate of convergence and the limiting distribution depend heavily on the value of $\mu$. Considering the cases at higher granularity, there is a phase transition from Gaussian to a more complicated limiting distribution as $|\mu|$ varies from a constant to zero. Such a precise phase transition can be fleshed out explicitly by considering a Berry--Esseen bound in the Kolmogorov distance for $U_n$ which should remain valid uniformly over all values of $|\mu|$. 
\newline 
\newline 
Unfortunately, such a uniform Berry--Esseen bound for $U$-statistics is unavailable. Most existing BE bounds either assume explicitly a non-degenerate or degenerate kernel or do not imply convergence in distribution under degeneracy. For example, the BE bounds of~\cite{bentkusalberink} are of the order $n^{-1/2}\mathbb{E}[|H_1(X_1) - \theta|^3]/(\mbox{Var}(H_1(X)))^{3/2}$, which becomes order 1 if $\mbox{Var}(H_1(X))$ is close to zero. On the other hand, the BE bounds of~\cite{yanus} and~\cite{huanggandy} do not assume explicit non-degeneracy but their bounds only converge to zero at a rate of $n^{-1/12}$ or $n^{-1/14}$, with some dependence on the eigenvalues of the kernel. The BE and Edgeworth expansions bounds of~\cite{Bentkus1999} could imply an $n^{-1/2}$ rate but requires too many eigenvalues of the kernel to be non-zero (e.g., their result would not apply for the estimation of $\mu^2$). Currently, there is no BE bound that provides a ``smooth'' phase transition between Gaussian and non-Gaussian limiting distributions. Here we also mention the works~\cite{bentkus1996optimal}, where a BE bound of order $n^{-1}$ is obtained for quadratic forms for $\mathbb{R}^d$- or Hilbert-space valued random variables. These quadratic forms can be viewed as degenerate second-order $U$-statistics, and the results of~\cite{bentkus1996optimal} can be contrasted with the rates obtained in~\cite{bentkusalberink} and~\cite{yanus}. We consider this paper as a first step in deriving a BE bound that gets the best of these settings and exhibits a smooth transition. We mention here the recent article~\cite{huang2024slow} that provides lower bounds on the BE bound for $U$-statistics. It should, however, be clarified that these are worst case lower bounds and do not contradict out goal of BE bounds that exhibit smooth transitions.
\newline 
\newline 
The derivation of BE bounds in the Kolmogorov distance involves two key steps: (1) a smoothing inequality; and (2) a L\'evy concentration function estimate. The smooth part of the smoothing inequalities revolve around the estimation of the characteristic function or thrice-differentiable functions of the $U$-statistic, and has been well-studied. The anti-concentration inequality is the subject of the current paper. The concentration function estimate from~\cite{bentkusalberink} would depend on $\mbox{Var}(H_1(X))$ and the anti-concentration inequality from~\cite{yanus} and ~\cite{huanggandy} does not depend on this variance, but on the eigenvalues of the kernel. However, if one combines the concentration function estimate of of~\cite{yanus} or~\cite{huanggandy} with the techniques of~\cite{bentkusalberink}, then one obtains a sub-optimal rate of at most $n^{-1/4}.$ We note here that the anti-concentration inequality is needed for the supposed limiting distribution and not the $U$-statistic itself. The goal of the current paper is to thus derive a concentration function estimate that smoothly transitions between Gaussian and chi-square regimes. The next section discusses the overall organization of the article. 
\subsection{Organization}
Here we outline the various sections of our work.
     \textbf{Section \ref{sec:limi-dist-U-stat}} describes how this work relates to quantitative limit theorems for U-statistics. If one is familiar with the general machinery for proving Berry-Esseen bounds in the Kolmogorov distance and limit theory for second order $U$-statistics, this section can be skipped entirely. \textbf{Section \ref{subsec:lit-review}} contains a review of current literature on concentration function estimates of quadratic forms, as well as a discussion of how and when our bound provides improvements. \textbf{Section \ref{sec: Notation}} is a guide to the notation used for the remainder of the note. \textbf{Section \ref{Sec:Estimate}} contains our main estimate and a discussion of its features and drawbacks as well as a comparison to the current estimates in the literature. The section also contains an overview of its proof. \textbf{Section \ref{Sec: Proof}} contains the proof of our estimate, assuming a technical lemma (Lemma \ref{Lem:LambdaIntEst}) which is proved in the appendix, see \ref{prf:LambdaIntEst}. \textbf{Section \ref{Sec:Applications}} provides some applications of our bound to certain U-statistics. \textbf{Section \ref{Sec: Conclusion}} is a brief conclusion and discusses some natural directions for future work. After this, our appendix contains proofs of a variety of results. The most significant are the proof of Lemma \ref{Lem:LambdaIntEst}, as well as a proof of how Theorem \ref{Thm:User Friendly Bound} is implied by Theorem \ref{Thm:Full Bound}. It also contains proofs of a lower bound for concentration functions, as well as proofs that our bounds match those currently available in the literature. We also discuss other techniques  available in the literature for estimating concentration functions, and how they failed to capture the behavior we desired. Finally, it contains detailed computations of the eigenvalues and eigenfunctions associated with the $U$-statistics mentioned in the applications section \ref{Sec:Applications}. 
\section{Distributional approximations for $U$-statistics}\label{sec:limi-dist-U-stat}
\subsection{Distributional approximations and the concentration function}
\label{sec:Asymp Dist U}
Let $X, X',X_1, \cdots, X_n$ be random IID variables taking values in a measurable space $(\mathcal{X}, \mathcal{A})$ with distribution $P_n$ (allowing for a triangular array setting.) Let $h:\mathcal{X} \times \mathcal{X} \to \bb R$ be a symmetric, square integrable, mean zero function with respect to $P_n\otimes P_n$, that is,
    $$
    h(x,x') = h(x',x), \quad \E[h(X, X')^2] < +\infty, \quad \E[h(X,X')] = 0 
    $$
The random variable 
\[
U_n = \frac{1}{\binom{n}{2}}\sum_{1\le i < j \le n} h(X_i, X_j),
\]
is known as the \textit{U-Statistic} associated to the \textit{kernel} $h$.
Lemma 5.1.5A of~\cite{serfling2009approximation} (known more generally as the \textit{Hoeffding Decomposition}) implies that we may decompose $U_n$ into two parts:
\begin{equation}
    \label{Eq: Hoeffding Decomp}
U_n = \frac{2}{n}\sum_{i=1}^n H_1(X_i) + \frac{1}{\binom{n}{2}}\sum_{1\le i < j\le n} H_2(X_i, X_j),
\end{equation}
where $$ 
H_1(x) = \mathbb{E}[h(X_1, X_2)|X_1 = x], \quad H_2(x, y) = h(x, y) - H_1(x) - H_1(y)$$ 
The first term in~\eqref{Eq: Hoeffding Decomp} is a sum of independent mean-zero random variables and is often known as the \emph{linear} part. The second can be thought of as an interaction term that remains after removing first order effects. The function $H_2(\cdot, \cdot)$ defines a self-adjoint Hilbert-Schmidt Operator $T:L^2(\mathcal{X}, P_n) \to L^{2}(\mathcal{X}, P_n)$ given by:
$$
Tf(x) := \int_{\mathcal{X}} H_2(x,y)f(y) \d P_n(y) = \E[H_2(x,X)f(X)]
$$
Consequently, the spectral theory of compact operators on Hilbert Spaces gives us that $H_2(\cdot, \cdot)$ may be written in the form:
\begin{equation}\label{eq:eigen-decomposition}
H_2(x, y) = \sum_{k=1}^{\infty} b_k\phi_k(x)\phi_k(y), \quad \text{(In the $L^2(P_n \otimes P_n)$ sense)}
\end{equation}
where $\{\phi_k\}_{k\ge1}$, $\phi_k:X \to \mathbb{R}$ is a sequence of \textit{orthonormal eigenfunctions} of $T$ i.e.,
\[
\mathbb{E}[\phi_j(X_i)\phi_k(X_i)] = \begin{cases}1, &\mbox{if }j=k,\\
0, &\mbox{if }j\neq k.
\end{cases}
\]
and the functions $\{\phi_k\}_{k\ge1}$ and $\{b_k\}_{k\ge1}$ satisfy
\[
b_k \phi_k(x) = (T\phi_k)(x) = \mathbb{E}[H_2(x, X)\phi_k(X)],
\]
and $\sum_{k=1}^{\infty} b_k^2 < \infty$. We may, without loss of generality, assume that $|b_1| \geq |b_2| \geq |b_3| \geq \cdots$ after appropriate rearrangement. An important fact to notice is:
$$
(T \mathbbm{1})(x) = \int_{\mathbb{R}}H_2(x,y)\d P_n(y) =  \E[H_2(x,X')] = 0.
$$
This means that $\mathbbm{1}$ is always an eigenfunction of $T$ with eigenvalue $0$. Thus, a corollary of orthonormality is that $\E[\phi_k(X)] = \bb{E}[\phi_k(X) \mathbbm{1}] = 0$ for any eigenfunction with $1\leq k < + \infty$. Note that in our previous terminology this states that the $U$-statistic associated to the kernel $H_2$ is always \emph{degenerate}.
\newline 
\newline 
With all of this in mind, let us jointly consider $H_1, H_2$. If needed, let $\phi_0(\cdot)$ be a function orthonormal to $\{\phi_k(\cdot): 1\leq k \leq \infty\}$ such that $\{\phi_k(\cdot):\,0\leq k \leq +\infty\}$ forms an orthonormal basis of the linear span of $\{H_1(\cdot), E[H_2(\cdot, X)g(X)]: g\in L_2\}$. In words: this set spans the image of the operator $T$ (which is spanned by $\phi_k$ with indices $1 \leq k <+\infty$) along with the one-dimensional subspace spanned by $H_1$. We thus obtain the decompositions:
\[
H_1(x) = \sum_{k=0}^{\infty} a_k\phi_k(x),\quad\mbox{and}\quad H_2(x, y) = \sum_{k=0}^{\infty}b_k \phi_k(x)\phi_k(y).
\]
for some square summable coefficients $a_k$, determined by the standard $L^2(P_n)$ inner product, i.e. $a_k = E[H_1(X)\phi_k(X)]$ and $b_0 = 0$. Substituting this into the Hoeffding decomposition \eqref{Eq: Hoeffding Decomp}, $U_n$ can be represented as:
\begin{align*}
U_n &= \frac{2}{n}\sum_{i=1}^n \left(\sum_{k=0}^{\infty} a_k\phi_k(X_i)\right) + \binom{n}{2}^{-1}\sum_{1\le i<j\le n} \left(\sum_{k=0}^{\infty} b_k\phi_k(X_i)\phi_k(X_j)\right)\\
&= \sum_{k=0}^{\infty} \left[\frac{2a_k
}{n}\sum_{i=1}^n \phi_k(X_i) + \frac{2b_k}{n(n-1)}\sum_{1\le i<j\le n} \phi_k(X_i)\phi_k(X_j)\right].
\end{align*}
Using the identity
\[
\frac{2}{n}\sum_{1\le i < j\le n} \xi_i\xi_j = (n-1)\left(\frac{1}{n}\sum_{i=1}^n \xi_i\right)^2 - \left(\frac{1}{n}\sum_{i=1}^n \xi_i^2 - \left(\frac{1}{n}\sum_{i=1}^n \xi_i\right)^2\right),
\]
and setting
\[
\overline{\phi}_k := \frac{1}{n}\sum_{i=1}^n \phi_k(X_i)\quad\mbox{and}\quad \overline{S}_{\phi_k}^2 := \frac{1}{n}\sum_{i=1}^n \phi_k^2(X_i) - \left(\frac{1}{n}\sum_{i=1}^n \phi_k(X_i)\right)^2,\quad k \ge 0,
\]
we obtain
\begin{align*}
U_n &= \sum_{k=0}^{\infty} \left[2a_k\overline{\phi}_k + \frac{b_k}{n-1}\left\{(n-1)\overline{\phi}_k^2 - \overline{S}_{\phi_k}^2\right\}\right]\\
&= \sum_{k=0}^{\infty} \left[\frac{2a_k}{\sqrt{n}}(n^{1/2}\overline{\phi}_k) + \frac{b_k}{n}(n^{1/2}\overline{\phi}_k)^2 - \frac{b_k}{n-1}\overline{S}_{\phi_k}^2\right]\\
&= \sum_{k=0}^{\infty} \left[\frac{2a_k}{\sqrt{n}}(n^{1/2}\overline{\phi}_k) + \frac{b_k}{n}\left\{(n^{1/2}\overline{\phi}_k)^2 - 1\right\} - \frac{b_k}{n-1}\left\{\overline{S}_{\phi_k}^2 - 1\right\} - \frac{b_k}{n(n-1)}\right].
\end{align*}
Heuristically, we have that the sequence $(n^{1/2}\overline{\phi}_k)_{k\ge0}$ converges in distribution to an infinite-dimensional Gaussian random variable with mean zero and identity covariance, because $\mathbb{E}[\phi_k(X)] = 0$, $\mathbb{E}[\phi_k^2(X)] = 1$, and $E[\phi_k(X)\phi_{k'}(X)] = 0$ for $k\neq k'$. This, in turn, suggests that the distribution of $U_n$ should be approximately one where $n^{1/2}\overline{\phi}_k, k\ge 0$ are replaced with an independent sequence of standard Gaussian random variables. More formally, consider $\{G_{i,k}, 1\le i\le n, k\ge0\}$, a set of iid $\Normal(0, 1)$ random variables. Define
\[
Z_k = \frac{1}{\sqrt{n}}\sum_{i=1}^n G_{i,k},\quad\mbox{and}\quad S_k^2 = \frac{1}{n}\sum_{i=1}^n G_{i,k}^2 - \left(\frac{1}{n}\sum_{i=1}^n G_{i,k}\right)^2.
\]
This gives the approximating random variable as:
\[
W_n = \sum_{k=0}^{\infty} \left[\frac{2a_k}{\sqrt{n}}Z_k + \frac{b_k}{n}\left\{Z_k^2 - 1\right\} - \frac{b_k}{n-1}\left\{S_k^2 - 1\right\} - \frac{b_k}{n(n-1)}\right].
\]
One can make this argument rigorous by considering a truncated series in the definition of $U_n$ to get $U_{n,K}$ defined as
\[
U_{n,K} ~:=~  \sum_{k=0}^{K} \left[\frac{2a_k}{\sqrt{n}}(n^{1/2}\overline{\phi}_k) + \frac{b_k}{n}\left\{(n^{1/2}\overline{\phi}_k)^2 - 1\right\} - \frac{b_k}{n-1}\left\{\overline{S}_{\phi_k}^2 - 1\right\} - \frac{b_k}{n(n-1)}\right].
\]
Then approximating $U_{n,K}$ with $W_{n,K}$ (the approximate version of $W_n$) and then letting $K$ diverge slowly enough, one can prove that that the Kolmogorov-Smirnov distance between $U_n$ and $W_n$ converges to zero; see, for example, Section 5.5.2 of~\cite{serfling2009approximation} for details.
\newline 
\newline 
To quantify the accuracy of this approximation, we seek an estimate of
the Kolmogorov-Smirnov distance:
\[
d_{\mathrm{KS}}(U_n, W_n) = \sup_{t}|\mathbb{P}(U_n \le t) - \mathbb{P}(W_n \le t)|.
\]
Note that for any choice of smooth functions $g_{t,\pm\varepsilon}(\cdot)$ satisfying
\[
\mathbf{1}\{x \le t - \varepsilon\} \le g_{t,-\varepsilon}(x) \le \mathbf{1}\{x \le t\} \le g_{t,\varepsilon}(x) \le \mathbf{1}\{x \le t + \varepsilon\},
\]
we have
\begin{equation}\label{eq:Lindeberg-smoothing-inequality}
d_{\mathrm{KS}}(U_n, W_n) \le \sup_{t,\eta\in\{\pm\varepsilon\}}\,|\mathbb{E}[g_{t,\eta}(U_n)] - \mathbb{E}[g_{t,\eta}(W_n)]| ~+~ \sup_{t\in\mathbb{R}}\,\mathbb{P}(t - \varepsilon \le W_n \le t + \varepsilon).
\end{equation}
The second term here is related to the L{\'e}vy concentration function of $W_n$: for any random variable $Y$, define
\[
Q_Y(l) ~:=~ \sup_{a, b\in\mathbb{R}:\, b - a \le l}\,\mathbb{P}(a < Y \le b).
\]
It should be mentioned here that the smoothing inequality~\eqref{eq:Lindeberg-smoothing-inequality} could be replaced with the Esseen smoothing inequality~\citep[Theorem 5.3]{Petrov}, where the first term of~\eqref{eq:Lindeberg-smoothing-inequality} is replaced with a term involving the difference of characteristic functions of $U_n$, and $W_n$. \cite{bentkus1996optimal},~\cite{yanus}, and others provide bounds on the first part of these smoothing inequalities. The focus in the current paper is the concentration function. 
\newline 
\newline 
By standard results on Gaussian random variables, we know that $S_k^2$, $Z_k$ are independent, and moreover, if $X+Y$ are independent, then $Q_{X+Y}(\e) \leq Q_X(\e)$~\citep[Lemma 1.11]{Petrov}. Thus, if we define:
\begin{equation}\label{eq: Adjusted Limit Stat}
\widetilde{W}_n := \sum_{k=0}^\infty \left[\frac{2a_k}{\sqrt{n}}Z_k + \frac{b_k}{n}\left\{Z_k^2 - 1\right\}\right] = \sum_{k=1}^{\infty} \sqbrace{\frac{2a_{k-1}}{\sqrt{n}}Z_k + \frac{b_k}{n}(Z_k^2 -1)},
\end{equation}
then $Q_{W_n}(\e) \le Q_{\widetilde{W}_n}(\e)$. Thus it suffices to estimate $Q_{\widetilde{W}_n}(\e)$. In particular $Q_{\widetilde{W_n}}$ is exactly of the form of the statistic $S$ mentioned in the abstract, with $\mu_k = \frac{2a_{k-1}}{\sqrt{n}}$ and $\l_k = \frac{b_k}{n}$ for $k \geq 1$.
\newline 
\newline 
While we motivated the problem of bounding the concentration function from $U$-statistics, there are other applications that involve similar concentration functions; see, for example,~\cite{gotze2019large}. Before proceeding to present our results on the concentration function, we review existing results and their limitations.

\subsection{Literature Review}\label{subsec:lit-review}
Let $\{Z_k\}_{k =1}^\infty$ be an infinite sequence of independent standard Gaussian random variables on some common probability space, and let $\{\lambda_{k}\}_{k=1}^{\infty}, \{\mu_{k}\}_{k=1}^\infty$ be real-valued, square-summable sequences with $|\l_k| \geq |\l_{k+1}| $. We study the concentration function of
\[
S := \sum_{k=1}^{\infty} \left[\l_k(Z_k^2 - 1) + \mu_kZ_k\right].
\]
Here the infinite series is to be interpreted as an almost-sure limit, which by the Kolmogorov three-series theorem is well-defined if $\norm{\l}{2} + \norm{\mu}{2} < \infty$. We derive upper bounds on the concentration function:
$$
Q_S(\e) = \sup_{x \in \mathbb{R}}\mathbb{P}\paren{x < S \leq x+ \e }.
$$
The behavior of $Q_S(\e)$ was studied in~\cite{Bentkus1999}, corresponding to their Theorem 1.2. They provide a rate for statistics of a more general form, but their bound is finite only when $\lambda_9 \neq 0$. It was further studied in \cite{AlberinkBentkus1999}, who obtain a bound for more general statistics. However, their bound does not tend to $0$ as $\e \to 0^+$ as they work in a very general case in which the random variables involved may not be continuous. Finally, $Q_S(\e)$ was also studied in~\cite{bobkov2020two} under the assumption that $\l_k \ge 0$ and $\mu_k = 0$ for all $k \ge 1$. Note that, in the context of $U$-statistics, there are no constraints on the signs of $\lambda_k,\, k\ge 0$.  
While the theory of quadratic forms could be considered the ``first-step'' beyond that of sums of independent random variables, concentration function bounds for such objects can be rather subtle. This is because they are inherently ``mixed'' objects: they combine quadratic Gaussian (chi-square) and Gaussian behavior in one statistic and exhibit properties of both simultaneously. For the simplest (and perhaps most extreme) example, if $Z$ is a Gaussian random variable and $a,b \in \mathbb{R}$, then for $X_{a,b} := aZ^{2} + bZ$, we have $Q_{X_{0,b}}(\e)\asymp {\e}/{|b|}$ while $Q_{X_{a,0}}(\e) \asymp {\e^{1/2}}/{|a|}$. Thus, the individual coefficients on each term in the sum are of importance: e.g. if $|a|$ is small, but nonzero, and $|b|$ is large, one should expect a rate ``close'' to $\e$, and vice-versa for the opposite case. The bound used in ~\cite{yanus} is of the form $\e^{1/2}/|a|$ for all cases, whereas ~\cite{huanggandy} use the Carberry Wright inequality which provides an estimate of the form $\e^{1/2}/\mathrm{Var}(S)$. The upshot here is that these bounds \textit{always} yield a rate of order $\e^{1/2}$, even in cases when the linear part of $S$ is large compared to its quadratic part.
\newline 
\newline 
Currently known concentration function bounds for quadratic forms, while elegant and sharp in certain cases, often make some restrictive assumptions regarding the coefficients $\l_k$ and $\mu_k$, mostly to avoid working with degenerate cases as exhibited above. Matching, two-sided bounds are available when $\mu \equiv 0$ and $\l_1,\l_2 > 0$~\citep[Theorem 1]{bobkov2020two}. However, the presence of nonzero $\mu$ and differently signed $\lambda_k$ \textit{significantly} complicates the analysis. Theorem 2 of \cite{bobkov2020two} bounds $Q_S(\cdot)$ when $\mu \neq 0$, but only when all $\l_k$ are positive and satisfy a specific balancing condition. This condition makes quantitative the qualitative idea that better anti-concentration occurs when one summand does not dominate the others. Such ideas are closely related to the theory of majorization; see \cite{TkoczSchurEntropy} and \cite{MarshallOlkin} for a further discussion of this topic in the context of concentration function estimates. 
\newline 
\newline
Theorem 2 of \cite{bobkov2020two} is sharp in the sense that under the assumed balancing and non-negativity conditions of the theorem statement, the bound provided is an upper and lower bound up to universal constant factors. However, in the context of $U$-Statistics, the coefficients $\lambda_k$ correspond to the eigenvalues of a Hilbert-Schmidt operator on a Hilbert Space, while the $\mu_k$ have a similar interpretation. In particular $\lambda_k, \mu_k$ may be \textit{any} square-summable sequences. They need not satisfy any non-negativity or balancing conditions, as will be illustrated in Section~\ref{Sec:Applications}. This makes the existing concentration function bounds are either impossible to apply or sub-optimal.
\newline 
\newline
In summary, current results in the literature have obtained concentration function estimates in the following three settings:
 \begin{enumerate}
     \item Purely Gaussian random variables.
     \item Weighted sums of central chi-square random variables.
     \item Weighted sums of non-central chi-square random variables under a set of balancing and non-negativity assumptions.
 \end{enumerate}
We bound the concentration function $Q_S(\cdot)$ under minimal assumptions on the coefficients $\lambda_k$ and $\mu_k$. In particular, we make no assumptions on $\lambda_k$, $\mu_k$ beyond square-summability. Furthermore, our bounds are adaptive, in that they recover the right behavior as $\mu$ tends to $0$ (purely chi-square) and as $\lambda$ tends to $0$ (purely Gaussian). Our bounds extend currently existing bounds to new coefficient regimes, while recovering previously known bounds in already studied cases (except one). 
 
\section{Notation}\label{sec: Notation}
Throughout, we let $\lambda := (\lambda_{k})_{k=1}^{\infty}$ and $\mu := (\mu_{k})_{k=1}^\infty$ denote elements of $\ell^2(\N, \bb R)$ with decreasing absolute value, e.g. $|\lambda_1| \geq |\lambda_2| \geq  |\lambda_3| \geq \cdots$. For a sequence $x \in \mathbb{R}^{\mathbb{N}}$, we say $x \geq 0$ if $x_j \geq 0$ for all $j \in \mathbb{N}$ and $x \not \geq 0$ if this is not the case. For any $j \in \bb N$,  we denote:
$$\Lambda_j^2 := \sum_{k=j}^\infty  \lambda_k^2.$$ A special case is of course that $\L_1 = \norm{\l}{2}$. The symbols $p_j, q_j$  denote real valued sequences, satisfying the relations: 
\[
\frac{1}{p_j} := \frac{\l_j^2}{\L_1^2},\quad\mbox{for}\quad j\ge 1 \quad \text{and} \quad \frac{1}{q_j} := \frac{\l_j^2}{\L_2^2}\quad\mbox{for}\quad j\ge2. 
\]
In particular, note that $\sum_{j \ge 1} {1}/{p_j} = \sum_{j \ge 2}{1}/{q_j} = 1$. The sequence $\{Z_k\}_{k=1}^\infty$ is an infinite sequence of independent standard Gaussian random variables defined on some common probability space $(\Omega, \mathcal{F}, \bb P)$, while $\Phi$ denotes the standard Gaussian cumulative distribution function. Finally, $\phi_X(t)$ denotes the characteristic function of a random variable, given by $\phi_X(t) := \bb E [e^{itX}]$ whereas $f_X(x)$ denotes the density of a random variable $X$ whose distribution is absolutely continuous (with respect to the Lebesgue measure on $\bb R$). 
\section{The Main Estimate and its Properties}
\label{Sec:Estimate}
In this section, we present our main results and compare them to existing results. 

\begin{thm}[General Concentration Function Estimate]
\label{Thm:Full Bound}
    Let $S := \sum_{k=1}^{\infty}\lambda_k(Z_k^2 -1) + \mu_kZ_k$. 
    Then, the series defining $S$ converges almost surely and 
    \begin{align}
    Q_{S}(\e)
&\leq 35 \e \sqbrace{ \paren{\frac{1}{\L_1^2 + \norm{\mu}{2}^2}}^{1/2}  + \mathbbm{1}{\{\e \leq 2 \L_1\}}\exp\paren{-\frac{1}{4}\frac{\norm{\mu}{2}^2}{\L_1^2}}\mathcal{I}(\e, \lambda)} ,
\end{align}
where $\mathcal{I}(\e, \lambda)$ is given by
\begin{equation}
\mathcal{I}(\e,\l) ~=~ \begin{dcases}
    \frac{1}{\L_1} A_{p,1}^{1/p_1}A_{p,2}^{1/p_2}, & p_1 > 2, 
    \\ 
    \frac{1}{\L_1}\log\paren{\frac{2\L_1}{\e}}^{\frac{1}{2} + \frac{1}{p_2}}, &p_1 = 2, 
    \\  
                \frac{1}{\L_1}B_{p,1}^{1/p_1}A_{p,2}^{1/p_2}, & 1 \leq p_1 < 2,  \\
   \frac{1}{(\L_1\L_2)^{1/2}}\paren{1   +  \mathbbm{1}{\{\e \leq 2 \L_2\}C_{q,2}}}, 
     & 1 \leq p_1 < 2  ,\ q_2 > 1,
\end{dcases}
\end{equation}
with
    \begin{align}
        &A_{p,i}= \paren{\frac{1}{\frac{1}{2} - \frac{1}{p_i}}}\sqbrace{1 - \paren{\frac{\e}{2\L_1}}^{\frac{p_i}{2} - 1}}, \\ 
        &B_{p,i}=\paren{\frac{1}{\frac{1}{p_i} - \frac{1}{2}} }\sqbrace{\paren{\frac{\e}{2\L_1}}^{\frac{p_i}{2} - 1} - 1}, \\ 
        &C_{q,i} = \paren{\frac{1}{1 - \frac{1}{q_i}}}\sqbrace{1 - \paren{\frac{\e}{2\L_2}}^{\frac{q_i}{2} - \frac{1}{2}}} .
    \end{align}
\end{thm}
\begin{thm}[User-friendly Concentration Function Estimate] Given the same setup as Theorem \ref{Thm:Full Bound}, we have the estimate:
\label{Thm:User Friendly Bound}

\begin{align}
Q_{S}(\e)
&\leq 35 \e \sqbrace{ \paren{\frac{1}{\L_1^2 + \norm{\mu}{2}^2}}^{1/2}  + \mathbbm{1}{\{\e \leq 2 \L_1\}}\exp\paren{-\frac{1}{4}\frac{\norm{\mu}{2}^2}{\L_1^2}}\mathcal{J}(\e, \lambda)},
\end{align}
where $\mathcal{J}(\e, \lambda)$ is given by
\begin{equation}
\mathcal{J}(\e, \l) ~:=~
\begin{dcases}
    \frac{1}{\L_1}\frac{1}{\frac{1}{2} - \frac{1}{p_1}} & p_1 > 2, 
    \\ 
    \frac{1}{\L_1}\log\paren{\frac{2\L_1}{\e}}^{\frac{1}{2} + \frac{1}{p_2}} &p_1 = 2,
    \\ 
    \frac{1}{\L_1^{\frac{3}{2} - \frac{1}{p_1}}}\frac{1}{\frac{1}{p_1} - \frac{1}{2}} \frac{1}{\e^{\frac{1}{p_1} - \frac{1}{2}}}
     & 1 \leq p_1 < 2,  \\
    \frac{1}{(\L_1\L_2)^{1/2}}\frac{1}{1- \frac{1}{q_2}} & 1 \leq p_1 < 2, \ q_2 > 1.
\end{dcases}
\end{equation}
\end{thm} 
The only difference between Theorem \ref{Thm:User Friendly Bound} and Theorem \ref{Thm:Full Bound} is the replacement of the simpler, easier to understand $\mathcal{J}(\e,\l)$ with the more complicated but ultimately more accurate, and sometimes better-behaved $\mathcal{I}(\e, \l)$. It is worth noting that our piecewise expressions are not disjoint; we obtain a bound for the case $1 \leq p_1 < 2$ (valid for all sub-cases) as well as the case $1 \leq p_1 < 2$ and $q_2 > 1$. These bounds are different, but both are valid, and so one may take the minimum of the two in applications. Finally, the constant 35 is by no means optimal and can definitely be reduced if one tracks constants more carefully through the course of the proof.
\newline 
\newline 
The difference between $\mathcal{I}$ and $\mathcal{J}$ is most pronounced in certain critical/edge cases, as we will see in the next section. The proof of Theorem~\ref{Thm:Full Bound} follows from a classical result connecting the concentration function to the characteristic function: Lemma 1.16 of~\cite{Petrov} states that for any random variable $Y$,
\begin{equation}\label{eq:Petrov-anti-concentration}
Q_Y(\e) \le \left(\frac{96}{95}\right)^2\max\{\e, 1/a\}\int_{-a}^{a} |\phi_Y(t)|dt,\quad\mbox{for all}\quad \e \ge 0, a \ge 0.
\end{equation}
Independence of the $Z_k, k\ge 1$ implies a product structure for $\phi_S(t)$, but controlling the integral is subtle depending on the cases listed in Theorem~\ref{Thm:Full Bound}. The full proof can be found in Section \ref{Sec: Proof}, with an outline in Section \ref{ProofTech}. It should be noted that the proof approaches of~\cite{bobkov2020two} and \cite{gotze2019large} are similar to ours.
The proof that Theorem \ref{Thm:Full Bound} implies Theorem \ref{Thm:User Friendly Bound} is brief and may be found in the appendix; see Section~\ref{Pf: Thm 1 Implies Thm 2}. 
\newline 
\newline 
Given that $S$ can be viewed as a sum of infinitely many independent random variables and also as a quadratic form of independent standard Gaussian random variables, it is possible to apply general tools such as one of the elegant Kolmogorov-Rogozin  Carbery-Wright inequalities. However, neither of them yield a bound with the correct behavior in all cases and they sometimes provide a bound with significant slack; we defer the details to Sections \ref{Sec: KR Ineq} and \ref{Sec: CR Ineq}. It is worth noting that the Kolmogorov-Rogozin inequality can be derived from~\eqref{eq:Petrov-anti-concentration}, which is the main result we use directly in our proof - thus the slack in the bound it provides is not entirely unexpected.
\newline 
\newline
Before comparing our results to the existing ones, we first consider how random variable $S$ behaves for different choices of $(\lambda_k,\,\mu_k),\,k\ge 1$.
\begin{enumerate}[label=(S\arabic*)]
    \item If $\Lambda_1 = \|\lambda\|_2 = 0$ (i.e., pure Gaussian), then $S\sim N(0, \|\mu\|_2^2)$ which implies that the maximum value of density of $S$ is $(2\pi)^{-1/2}/\|\mu\|_2$, i.e., $Q_S(\e) = (2\pi)^{-1/2}\|\mu\|_2^{-1}\e$.\label{setting1}
    \item If $\|\mu\|_2 = 0$ and $\Lambda_2 = 0$ (i.e., chi-square), then $S\sim \lambda_1(\chi_1^2 - 1)$, the density of which is unbounded and $Q_S(\e) \asymp \e^{1/2}/\lambda_1$ as $\e\to0^+$.\label{setting2}
    \item If $\|\mu\|_2 = 0$ and $\lambda_1 = \lambda_2 = 1$ with $\Lambda_3 = 0$, then $S$ is a sum of two independent centered chi-square random variables which has a bounded density (because Exp(1) has a bounded density) and one should expect $Q_S(\e) \asymp \varepsilon$ as $\varepsilon\to0^+$ (with constants depending on $\lambda_1, \lambda_2$). \label{setting3}
    \item If $\|\mu\|_2 = 0$ and $\lambda_1 = -\lambda_2 = 1$ with $\Lambda_3 = 0$, then $S$ is a difference of two independent chi-square random variables which does {\em not} have a bounded density and it can be proved that $Q_S(\e) \asymp \e\ln(1/\e)$ as $\e \to 0^+$; see Section~\ref{sec:concentration function lower bound} for a proof.\label{setting4}
\end{enumerate}
Comparing settings~\ref{setting3} and~\ref{setting4}, it is interesting to note that the mismatch in the signs of $\lambda_1, \lambda_2$ plays a significant role in the behavior of the concentration function. In relation, the results of~\cite{bobkov2020two} apply for~\ref{setting3} but not~\ref{setting4}. Our bound applies in all cases because our bound is sign agnostic. Given that lower bounds are in the settings discussed above, we expect a good bound on $Q_S(\e)$ to at least match the orders in these settings. Also, note that comparing settings~\ref{setting2} and~\ref{setting3}, we expect the behavior of the concentration function changes from $\varepsilon$ to $\varepsilon^{1/2}$ as $\lambda_2\downarrow 0$. Capturing this behavior is one of the reasons why $\mathcal{I}(\cdot, \cdot)$ and $\mathcal{J}(\cdot, \cdot)$ depend on the size of $p_1$ and $q_2$; see Section~\ref{subsec:boundary-behavior} for more details. This relates to balancing of the $\lambda$ coefficients. The importance of balancing of coefficients has been studied elegantly using the framework of majorization and Schur-concavity/monotonicty (cf.~\cite{MarshallOlkin}). This viewpoint is explored thoroughly in \cite{TkoczSchurEntropy} who obtain various results, in particular the maximum of the density, for weighted sums of independent gamma random variables.
\newline 
\newline
Our bound in Theorems \ref{Thm:Full Bound}, \ref{Thm:User Friendly Bound} can be written as a sum of two terms, $Q_S(\e) \leq  35\sqbrace{Q_{G}(\e) + \mathbbm{1}\{\e \leq 2\Lambda_1\}Q_C(\e)}$ where
\begin{equation}\label{eq:Gaussian-and-chi-square-anti-concentration}
Q_G(\e) ~:=~ \e  \paren{\frac{1}{\L_1^2 + \norm{\mu}{2}^2}}^{1/2},
\quad\mbox{and}\quad
Q_C(\e) ~:=~ \e\mathcal{I}(\e,  \lambda)\exp\paren{-\frac{1}{4}\frac{\norm{\mu}{2}^2}{\L_1^2}}.
\end{equation}
which can be interpreted as Gaussian and chi-square ``components'' of the concentration function. In Setting~\ref{setting1}, $Q_C(\e) = 0$ and $Q_G(\e) = \e/\|\mu\|_2$ matching the right behavior. In Setting~\ref{setting2}, $p_1 = 1$, $p_2 = \Lambda_1^2/\lambda_2^2 = \infty$ and hence, $\mathcal{J}(\e, \lambda) = 2\Lambda_1^{-1}/\e^{1/2}$ matching the right behavior. In Settings~\ref{setting3} and~\ref{setting4}, $\mathcal{J}(\e, \lambda) = \L_1^{-1}(\log(2\L_1/\e))$ by taking $p_1= 2$. Interestingly, our bound cannot recover the right behavior for Setting~\ref{setting3} because our proof technique cannot take advantage of matching signs of $\lambda_1, \lambda_2$. In particular, $|\phi_S|$ depends only on the absolute values of the $\lambda_k$.
\paragraph{Comparison with the bounds of \cite{bobkov2020two} and \cite{gotze2019large}.} 
Theorem 1 and Theorem 2 of \cite{bobkov2020two} provide matching upper and lower bounds on the density $f_S(\cdot)$ of $S$ under their stated assumptions. Paraphrasing the results in our notation, Theorem 1 of~\cite{bobkov2020two} covers the case when $\mu = 0$ and $\l \geq 0$, and states that $\norm{f_{S}}{\infty} \asymp {\paren{\L_1 \L_2}^{-1/2}}$, which implies that
    \begin{equation}
        Q_S(\e) \lesssim  \frac{\e}{\paren{\L_1 \L_2}^{1/2}}.
    \end{equation}
    Theorem $2$ of~\cite{bobkov2020two} covers the case when $\mu \neq 0$ and $\l \ge 0$, but assumes an additional "balancing" condition, which in our notation becomes
    \begin{equation}
        \frac{1}{p_1^2} \leq \frac{1}{3}.
    \end{equation}
    This condition for example, ensures that for example $|\l_1| \ge |\l_2| \ge |\l_3| > 0$ (but is stronger). Under this assumption, their Theorem $2$ states $\norm{f_S}{\infty} \asymp {(4 \L_1^2 + \norm{\mu}{2}^2)^{-1/2},}$ which implies that
    \begin{equation}\label{eq:mu-not-zero-Bobkov}
        Q_S\paren{\e} \lesssim \frac{\e}{(4 \L_1^2 + \norm{\mu}{2}^2)^{1/2}}.
    \end{equation}
    As mentioned before, these assumptions are sometimes too restrictive in applications, and we obtain bounds under relaxed conditions. We are able to relax the non-negativity assumption on $\l$, and obtain bounds for all values of ${p_1}$. Our results show that $S$ has a bounded density under more general assumptions -- in particular, under either of the following conditions
    \begin{enumerate}
    \item ${1}/{p_1} < 1/2$;
    \item ${1}/{q_3} < 1$, i.e., $\l_3 \neq 0$.
    \end{enumerate}
    Our results also recover~\eqref{eq:mu-not-zero-Bobkov}; see Section \ref{Sec: Recovering Bobkov} for a proof. On the other hand, sign-agnostic nature of our bounds implies that they \emph{cannot} recover Theorem 1 of \cite{bobkov2020two}. But we note that no anti-concentration bound that is sign-agnostic can recover Theorem 1 of~\cite{bobkov2020two} as illustrated from Settings~\ref{setting3} and~\ref{setting4} above. 

    \subsection{Behavior of the Estimates at Boundaries}\label{subsec:boundary-behavior}
    As our estimates for $\mathcal{I}$ and $\mathcal{J}$ are piecewise, it is natural to understand their behavior at the respective boundaries in their definition. For example, at $p_1 = 2$, we witness a logarithmic transition - can this be achieved by taking $p_1 \to 2$ in either of the bounds we obtain for $p_1 > 2$ or $p_1 < 2$. The answer is \emph{yes}, with the caveat that we must consider the more complicated $\mathcal{I}$ from Theorem \ref{Thm:Full Bound}. 
However, such contiguity does not hold for the bound in Theorem \ref{Thm:User Friendly Bound}, which is one of its main weaknesses. We provide details in the analysis below.
    \begin{enumerate}
        \item $p_1 > 2$ and $p_1 \to 2^+$: For $p_1 > 2$, Theorem~\ref{Thm:User Friendly Bound} yields a linear rate of the form: 
        \begin{equation}
        \frac{\e}{\L_1}\frac{1}{1 - \frac{2}{p_1}}.
        \end{equation}
        If $p_1 \to 2^+$, however, this simplified bound diverges. In such cases, it is important to use the bound in Theorem $1$, as it recovers the appropriate logarithmic rate given in the case $p_1 = 2$. Indeed, we have:
        $$
        \frac{1}{\L_1}\frac{1}{1 - \frac{2}{p_1}}    \sqbrace{1 - \paren{\frac{2\L_1}{\e}}^{\paren{1 - \frac{p_1}{2}}}} = \frac{1}{\L_1} \frac{p_1/2}{p_1/2 - 1} \sqbrace{1 - \paren{\frac{2\L_1}{\e}}^{\paren{1 - \frac{p_1}{2}}}} \underset{p_1 \to 2^+}{\to} \frac{1}{\L_1}\log\paren{\frac{2\L_1}{\e}}.
        $$
        \item $p_1 = 2$: In this case, we obtain a rate of the form \begin{equation}\e (\log\paren{2\L_1/\e})^{\frac{1}{2} + \frac{1}{p_2}}.
        \end{equation}
        This rate is suboptimal in the special case where $\l_1$ and $\l_2$ have the same sign, but is otherwise optimal; see Section~\ref{sec:concentration function lower bound}. In some cases, for example, when $p_1$ is very close to $2$, this bound is better than the simplified bound obtained in the prior case (from Theorem \ref{Thm:User Friendly Bound}). However, some calculus shows it is always inferior to the bound we obtain in $p_1 > 2$ in Theorem \ref{Thm:Full Bound}. 
        \item $1\leq p_1 \leq 2$: In this case, the bound is the minimum of two other bounds; each of which is preferable in a distinct scenario:
        \begin{enumerate}
             \item The first is of the following form:
             $$
             \e\paren{\frac{1}{\L_1^{\frac{3}{2} - \frac{1}{p_1}}}\frac{1}{\frac{2}{p_1} - 1} \frac{1}{\e^{\frac{1}{p_1} - \frac{1}{2}}}} =\frac{1}{\frac{2}{p_1} - 1}\paren{\frac{\e}{\L_1}}^{\frac{3}{2} - \frac{1}{p_1}}.
             $$
             and it holds regardless of the values of $q_2$ and $\L_2$. This bound has no coefficient of the form ${\L_2}^{-1}$ (which in this case is large due to the assumption of $p_1$ being small), so it does not explode as $\L_2 \to 0^+$. However, it always yields sublinear concentration as $3/2 - {1}/{p_1} < 1$. Thus, we sacrifice decay in $\e$ but remove the dependence on $\L_2$. Notice also in the definition of $\mathcal{I}(\cdot, \cdot)$ that as $p_1 \to 2^{-}$, we obtain a bound matching the $p_1 = 2$ case, which can be seen much like in the case $p_1 > 2$ by computing limits. 
            \item The second bound behaves in the opposite manner and is of the form:
            \begin{equation}
            \frac{\e}{(\L_1\L_2)^{1/2}}\frac{1}{1 - \frac{1}{q_2}}.
            \end{equation}
            
            It requires $q_2 > 1$ which is equivalent to assuming $\l_3 \neq 0$ (but is weaker than assuming $p_1 \geq 3$). We obtain a linear rate in $\e$ but at the cost of a term ${\L_2}^{-1}$ so that the bound diverges when $\L_2 \to 0$, or when $q_2 \to 1$. 
        \end{enumerate}
        Thus, for small $\L_2$ and large $\e$, the first bound is preferable. For small $\e$ and large $\L_2$, then the second is preferable. 
    \end{enumerate}
    The bound $\mathcal{J}(\cdot, \cdot)$ is best summarized by the following flowchart:
    
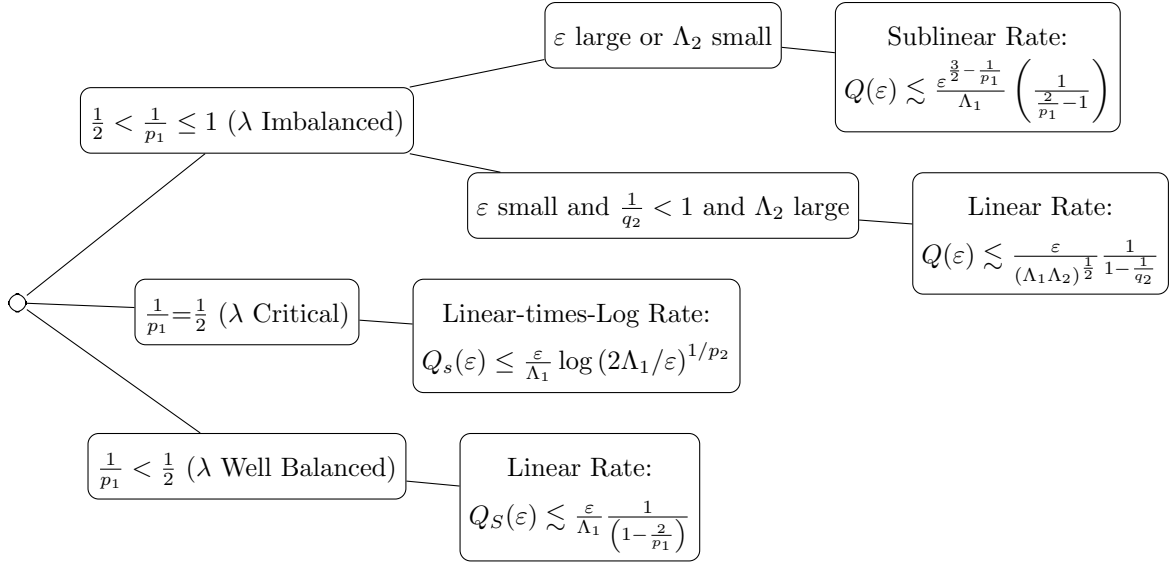
\begin{center}
    
    \begin{forest}
  for tree={math content,
    draw,
    rounded corners,
    align=center,
    edge={-},
    grow = east,
    l sep=20pt,
    s sep=15pt
  }
  [
    [$\frac{1}{p_1} < \frac{1}{2}$ ($\l$ Well Balanced)
      [Linear Rate:  \\ $ Q_S(\e) \lesssim\frac{\e}{\L_1} \frac{1}{\paren{1 - \frac{2}{p_1}}}$]
    ]
    [$\frac{1}{p_1} \mathrm{=} \frac{1}{2}$ ($\l$ Critical)
    [Linear-times-Log Rate: \\$Q_s(\e) \leq \frac{\e}{\L_1}\log
    \paren{2 \L_1/\e}^{1/p_2}$ ]
    ]
    [$ \frac{1}{2} < \frac{1}{p_1} \leq 1$ ($\l$ Imbalanced)
      [$\e$ small and $\frac{1}{q_2} < 1$ and $\L_2$ large
      [Linear Rate: \\ 
      $Q(\e) \lesssim \frac{\e}{(\L_1\L_2)^\frac{1}{2}} \frac{1}{1 - \frac{1}{q_2}}$]
      ]
      [$\e$ large or $\L_2$ small
      [Sublinear Rate: \\ $Q(\e) \lesssim \frac{\e^{\frac{3}{2} -\frac{1}{p_1}}}{\L_1} \paren{\frac{1}{\frac{2}{p_1} - 1}}$ ]
      ]
    ]
  ]
\end{forest}
\end{center}
    Although these balancing conditions may seem arbitrary, they stem from integrability properties of the modulus of the chi-square characterstic function $\prod_{k=1}^\infty {(1 + 4\l_k^2t^2)^{-1/4}}$. In cases where $p_1$ is large or $q_1$ is large, we can obtain linear rates if desired because our balancing conditions ensure that $\l_3 \neq 0$ and hence
        $$
        \prod_{k=1}^{\infty} \frac{1}{(1 + 4\lambda_k^2t^2)^{1/4}} \leq \frac{1}{4^{3/4}|\l_1 \l_2 \l_3|^{1/2}t^{3/2}}.
        $$
        which is integrable on $(a, +\infty)$ for any $a > 0$. However, using such a bound directly introduces an explicit dependence on the minimum eigenvalue, which we wish to avoid whenever possible. 
    
\section{Sketch of the Proof} \label{ProofTech}
As mentioned, our proof begins by invoking Lemma 1.16 of \cite{Petrov} as in equation \eqref{eq:Petrov-anti-concentration}. Fixing $\e > 0$ and setting $a = \e$ in this bound gives:
\begin{equation}
\label{eq: simplified petrov ac bound}
Q_X(\e) \leq \paren{\frac{96}{95}}^2 \e \int_{-\frac{1}{\e}}^{\frac{1}{\e}}|\phi_X(t)| \d t.
\end{equation}
If $X = \sum X_i$ is a sum of independent random variables, the right hand side is straightforward to compute as $\phi_{X}$ is simply the product of the individual $\phi_{X_i}$. In our case, a brief computation shows that:
$$
|\phi_{\lambda_k (Z^2-1) + \mu_k Z}(t)| = \ \frac{\exp\paren{-\frac{2\mu_k^2t^2}{1+4\lambda_k^2t^2}}}{(1+4\lambda_k^2t^2)^{1/4}},
$$
so that:
$$
|\phi_{S}(t)| = \frac{\exp\paren{-\sum_{k=1}^{\infty}\frac{2\mu_k^2t^2}{1+4\lambda_k^2t^2}}}{\prod_{k=1}^{\infty}(1+4\lambda_k^2t^2)^{1/4}}.
$$
To estimate this, we apply the following heuristics:
\begin{enumerate}
    \item For $t$ such that $4\lambda_k^2t^2$ is of constant order, we have:
$$
|\phi_S(t)| \approx \exp\paren{-t^2 \sqbrace{2\sum_{k=1}^\infty\mu_k^2}}.
$$
The right hand side of the above is the characteristic function of a Gaussian with variance $\norm{\mu}{2}^2$. Thus, for small $t$, the characteristic function is essentially that of a Gaussian. In the notation of equation \eqref{eq:Gaussian-and-chi-square-anti-concentration}, this gives us the $Q_G$ term. 
\item 
On the other hand, if $t$ is taken "large":
$$
|\phi_S(t)| \approx
\frac{\exp\paren{-\sum_{k=1}^{\infty} \frac{\mu_k^2}{4\l_k^2}}}{\prod_{k=1}^{\infty}(1+4\lambda_k^2t^2)^{1/4}} = C(\mu,\l) \prod_{k=1}^{\infty}\frac{1}{(1 + 4 \l_k^2t^2)^{1/4}}
$$
since $\frac{t^2}{1 + 4 \lambda_k^2t^2} \to \frac{1}{4\lambda_k^2} $ as $t\to \infty$. 
In this regime, the integrand behaves like that of the characteristic function of a weighted sum of independent chi-square random-variables, multiplied by a prefactor $C(\mu, \l)$ that decays exponentially as a function of the ratio $\frac{\norm{\mu}{2}^2}{\L_1^2}$. In the notation of equation \eqref{eq:Gaussian-and-chi-square-anti-concentration}, this gives us the term $Q_C$. 
\end{enumerate}
Adding these two terms and applying \eqref{eq: simplified petrov ac bound} gives
\begin{equation}
Q_X(\e) \lesssim \e \int_{-1/\e}^{1/\e} |\phi_X(t)|\d t \approx \e\sqbrace{\underbrace{\int_{t \text{ small}} \exp\paren{-t^2 \sqbrace{2\sum_{k=1}^\infty\mu_k^2}}\d t}_{\text{Gaussian Part}}+ C(\mu,\lambda)\underbrace{\int_{t \text{ large}}\frac{1}{\prod_{k=1}^{\infty}(1+4\lambda_k^2t^2)^{1/4}} \d t}_{\text{Chi Square Part}}}.
\end{equation}
The first integral is straightforward to compute, being the integral of a Gaussian function over some interval. Estimating the second integral is more involved and its behavior is highly dependent on the relative size of the $\lambda_k$. For example, if all $\lambda_k$ but $\lambda_1, \lambda_2$ are $0$, then the integral over the whole real line (corresponding to $\e \to 0^+$) will diverge, and we cannot obtain a linear in $\e$ with this technique. On the other hand, if $\lambda_1, \lambda_2, \lambda_3 \neq 0$ then said integral will converge and give us a linear rate. The estimation of this integral is the content of Lemma \ref{Lem:LambdaIntEst}: 
\begin{lemma} 
\label{Lem:LambdaIntEst}
    Let $\mathcal{I}$ be defined as in Theorem \ref{Thm:Full Bound}. Then:
    \begin{equation}
      \int_{\frac{1}{2\L_1}}^\frac{1}{\e} \prod_{k=1}^\infty \frac{1}{(1 + 4 \l_k^2 t^2)^{1/4}}\d t  \leq 6 \mathcal{I}(\e, \lambda).
    \end{equation}
\end{lemma}  
\begin{proof}
    See the computations in Section \ref{prf:LambdaIntEst}. 
\end{proof}
The proof of this estimate is lengthy but straightforward. Our approach follows that of Lemma A2 from \cite{gotze2019large}, which hinges upon an application of H\"older's inequality with an intelligent choice of weights. However, their proof only covers the case when $p_1 \geq 3$ or $\frac{{\l_1}^2}{\L^2} \leq \frac{1}{3}$. Our proof provides estimates for the remaining cases, based on casework depending on the relative sizes of $p_k$. The most involved case is showing we can obtain a linear rate when $\frac{1}{p_1} > \frac{1}{2}$ but $\l_3 \neq 0$, e.g. the final case in the definitions of $\mathcal{J}, \mathcal{I}$. This case represents coefficients that are poorly balanced yet still numerous enough to provide linear concentration in $\e$. The proof technique in this case is essentially a one-step recursion and was inspired by the proof of Case 2 in Theorem 1 of \cite{bobkov2020two}. 
    \section{Applications}
    \label{Sec:Applications}
    In this section, we apply Theorem \ref{Thm:Full Bound} and Theorem \ref{Thm:User Friendly Bound} to two interesting examples. Throughout this section, we repeatedly use the notation of Section \ref{sec:limi-dist-U-stat}, where we discuss the limiting behavior of U-statistics.
    \subsection{Square of the Mean and General Rank 1 - U-statistics}
    \label{sec: rank 1 application}
    Let $X, X_1, \cdots, X_n$ denote an IID sample from some probability distribution $P_n$ with $\E[X] = \theta_n, \mathrm{Var}(X) = \sigma_n^2$. We wish to obtain an unbiased estimate of $\theta_n^2$. The corresponding U-Statistic is:
    \begin{equation}
    \widetilde{U}_n = \frac{1}{{n \choose 2}} \sum_{1 \leq i < j \leq n} X_iX_j,
    \end{equation}
    and its centered version (which we consider hereafter):
    \begin{equation}
        U_n = \frac{1}{{n \choose 2}} \sum_{1 \leq i < j \leq n} (X_iX_j - \theta_n^2),
    \end{equation}
    i.e., according to the setup in Section \ref{sec:Asymp Dist U} , we have kernel function $h(x,y) = xy  -\theta_n^2$.
    \newline 
    \newline
    The statistic $U_n$ above may seem artificial - the square of the mean is not necessarily an everyday quantity of interest. However, after applying the diagonalization procedure discussed in \ref{sec:Asymp Dist U}, \emph{every} U-statistic with a rank $1$ kernel can be written in such a form after a deterministic transformation (diagonalization). Thus, it is a simple yet broadly applicable starting point for any analysis. In the appendix, we show explicitly how $U_n$ exhibits a phase transition in its distribution, based on the parameters $\theta_n, \sigma_n$ - see - The calculation is essentially identical to one done in Section \ref{sec:Asymp Dist U} but is more transparent. 
\newline 
\newline 
    Following the notation of Section \ref{sec:limi-dist-U-stat}, we obtain by calculations in Appendix \ref{Sec: Eigenvalue Calculations} that:
    $$
    H_1(x) = x\theta_n - \theta_n^2 ,\quad  H_2(x,y) = xy - \theta_n(x+y)
    $$
    with a single non-constant eigenfunction 
    $$
    \phi_1(x) = \frac{x - \theta_n}{\sigma_n},
    $$
    and coefficients:
    $$
    a_1 = \theta_n \sigma_n, \quad b_1 = \sigma_n^2.
    $$
    So in summary:
    $$
    H_1(x) = a_1\phi_1(x), \quad H_2(x,y) = b_1\phi_1(x)\phi_1(y).
    $$
    This represents the simplest possible limiting behavior for a $U$-statistic - $H_2$ has one nontrivial eigenfunction and eigenvalue, and is a rank-one operator. Thus, as derived in \ref{sec:limi-dist-U-stat} the corresponding Gaussian approximation is
    $$
    W_n = \frac{2a_1}{\sqrt{n}}Z_1 +  \frac{b_1}{n} (Z_1^2 -1) - \frac{b_1^2}{n-1}(S^2 - 1) - \frac{b_1}{n(n-1)}.
    $$
    As mentioned in equation \eqref{eq: Adjusted Limit Stat} of Section \ref{sec:limi-dist-U-stat}, it suffices to estimate the concentration function of:
    \begin{align*}
   \widetilde{W}_n = \frac{2a_1}{\sqrt{n}}Z_1  + \frac{b_1}{n} (Z_1^2 -1) &=  \frac{2 \theta_n \sigma_n}{\sqrt{n}}Z_1 + \frac{\sigma_n^2}{n}
    (Z_1^2 - 1).
    \end{align*}
    This implies
    \[
    \frac{n^{1/2}}{\sigma_n^2}\widetilde{W}_n = \frac{2\theta_n}{\sigma_n}Z_1 + \frac{1}{\sqrt{n}}(Z_1^2 - 1), \quad \frac{n}{\sigma_n^2}\widetilde{W}_n = \frac{2n^{1/2}\theta_n}{\sigma_n}Z_1 + (Z_1^2 - 1),
    \]
    whose behavior depends only on that of the ratio $\frac{\theta_n}{\sigma_n}$.
    Let us take a moment to understand the limiting behavior of the statistic $\widetilde{W_n}$ and its dependence on $\theta_n, \sigma_n$. 
    \begin{enumerate}
        \item If $\theta_n/\sigma_n \gg n^{-1/2}$ (i.e., $n^{1/2}\theta_n/\sigma_n \to \infty$) as $n\to\infty$, and  $\frac{\theta_n}{\sigma_n} \to C$ for some $C \neq 0$, then $n^{1/2}\widetilde{W}_n/\sigma_n^2$ behaves like a mean zero Gaussian with variance $C^2$. 
        \item if $\theta_n/\sigma_n \asymp n^{-1/2}$ (i.e., $n^{1/2}\theta_n/\sigma_n \to c \in (0,\infty)$) as $n\to\infty$, then $n\widetilde{W}_n/\sigma_n^2$ behaves like a sum of centered Gaussian and a  centered chi-square random variable.
        \item Finally, if $\theta_n/\sigma_n \ll n^{-1/2}$ (i.e., $n^{1/2}\theta_n/\sigma_n\to0$) as $n\to\infty$, then $n\widetilde{W}_n/\sigma_n^2$ behaves like a centered chi-square random variable.
    \end{enumerate}
   
    \paragraph{Concentration Function of Normalized $\widetilde{W_n}$:} As above, we are interested in the concentration functions of the random variables:
    $$
     \frac{n^{1/2}}{\sigma_n^2}\widetilde{W}_n = \frac{2\theta_n}{\sigma_n}Z_1 + \frac{1}{\sqrt{n}}(Z_1^2 - 1), \quad \frac{n}{\sigma_n^2}\widetilde{W}_n = \frac{2n^{1/2}\theta_n}{\sigma_n}Z_1 + (Z_1^2 - 1).
    $$
    This is the sum of a Gaussian and a single chi-squared random variable. Thus, following the notation of Theorems \ref{Thm:Full Bound}, \ref{Thm:User Friendly Bound}, this corresponds to the case $p_1 = 1$ in the definition of $\mathcal{I}, \mathcal{J}$. Computing the relevant quantities and plugging into our bounds from Theorem \ref{Thm:Full Bound}/ Theorem \ref{Thm:User Friendly Bound} (they agree in this case) gives:
    \begin{equation}
    \label{Eqn: Q_Wn sqrt n}
    Q_{\frac{n^{1/2}}{\sigma_n^2}\widetilde{W}_n}(\e) \leq 35 \e \sqbrace{ \paren{\frac{1}{\frac{1}{n} + 4\frac{\theta_n^2}{\sigma_n^2}}}^{1/2}  + \mathbbm{1}{\left\{\e \leq \frac{2}{n^{1/2}}\right\}}\exp\paren{-n \frac{\theta_n^2}{\sigma_n^2}}\frac{2n^{1/4}}{\e^{1/2}}} ,
    \end{equation}
    and
    \begin{equation}
            \label{Eqn: Q_Wn n}
            Q_{\frac{n}{\sigma_n^2}\widetilde{W}_n}(\e) \leq 35 \e \sqbrace{ \paren{\frac{1}{1 + 4n\frac{\theta_n^2}{\sigma_n^2}}}^{1/2}  + \mathbbm{1}{\left\{\e \leq 2\right\}}\exp\paren{-n \frac{\theta_n^2}{\sigma_n^2}}\frac{2}{\e^{1/2}}}.
    \end{equation}
    Let us return to the three cases we discussed above and confirm that our estimate behaves congruously. Throughout, we assume $\e > 0$ is fixed as $n \to +\infty$.
    \begin{enumerate} 
        \item If $\frac{n^{1/2}\theta_n}{\sigma_n} \to \infty$ as $n \to \infty$ and $\frac{\theta_n}{\sigma_n} \to C$, then for the first term:
        $$
        \paren{\frac{1}{\frac{1}{n} + 4\frac{\theta_n^2}{\sigma_n^2}}}^{1/2} \rightarrow \frac{1}{2C},
        $$
        whereas for the second term
        $$
        \exp\paren{-n \frac{\theta_n^2}{\sigma_n^2}}\frac{2n^{1/4}}{\e^{1/2}} \to 0,
        $$
        giving the estimate: $$Q_{n^{1/2}\frac{\widetilde{W_n}}{\sigma_n^2}}(\e) \leq \frac{35}{2C}\e.$$
        This is Gaussian concentration; up to constants.
        \item If $\frac{n^{1/2}\theta_n}{\sigma_n}\to c \in \mathbb{R}$, then we examine $Q_{\frac{n}{\sigma_n^2}\widetilde{W}_n}$. For the first term we obtain:
        $$
     \paren{\frac{1}{1 + 4n\frac{\theta_n^2}{\sigma_n^2}}}^{1/2} \to \paren{\frac{1}{1 + 4c^2}}^{1/2},
        $$
        whereas for the second:
        $$
                \exp\paren{-n \frac{\theta_n^2}{\sigma_n^2}}\frac{2}{\e^{1/2}} \to \exp\paren{-c^2}\frac{2}{\e^{1/2}}.
        $$
        This yields a combined bound:
        $$
        35\sqbrace{\e\paren{\frac{1}{1 + 4c^2}}^{1/2} + 2\exp(-c^2)\e^{1/2}},
        $$
        which is a mix of Gaussian and chi-square behavior, with the relative size of each term depending on $c$
        \item If $n^{1/2}\theta_n/\sigma_n\to 0$ then the same analysis as above holds except we may replace $c$ with $0$. This gives us a bound of the form:
        $$
        35\sqbrace{\e + 2 \e^{1/2}},
        $$
        which has Chi-square behavior in the leading order term as $\e \to 0^+$ $(\e^{1/2})$.
    \end{enumerate}
    
    \subsection{Student's T-Statistic Kernel}
    \label{sec: t stat analysis}
    The following example is of particular interest due to the appearance of both positive and negative $\l_k$. 
    \newline 
    \newline 
    Let $X, X_1, \cdots, X_n$ be IID  with $\E[X] = 0, \E[X^2] = 1$, $\E[X^3] = \gamma^3$, $\E[X^4] = \kappa^4 < \infty$ and let $h(x,y) = xy^2 +x^2y$. Then:
    $$
    U_n = \frac{1}{{n \choose 2}}\sum_{i < j} X_iX_j^2 +X_i^2 X_j.
    $$
    This statistic is closely related to Student's T-statistic and self-normalized sums in general and is one of the main examples discussed in \cite{bentkusalberink}. Without loss of generality, we will assume that $\gamma^3 \geq 0$ for the remainder of the analysis. The case $\gamma^3 < 0$ is entirely analogous.
    \newline 
    \newline 
    By the calculations of Section \ref{Sec: Eigenvalue Calculations} in the appendix, we obtain:
    $$
    H_1(x) = \E[xX^2 + x^2 X] = x,
    $$
    and:
    $$
    H_2(x,y) = h(x,y) - H_1(x) - H_1(y) = xy^2 + x^2y - x - y = x(y^2 -1) + y(x^2 - 1).
    $$
    For notational brevity, let us define:
    $$
    \xi = \kappa^4 - 1.
    $$
    In this new notation, the coefficients $a_k$ are given by:
    $$
        a_1 = \frac{\paren{2 \paren{1 + \frac{2\gamma^{3}}{\xi^{1/2}}}}^{1/2}}{2}, \quad a_2 = \frac{\paren{2 \paren{1 - \frac{2\gamma^{3}}{\xi^{1/2}}}}^{1/2}}{2}.
    $$
    Note that $a_1^2 + a_2^2 = 1$.
    \newline 
    \newline
    The coefficients $b_k$ are given by:
    $$
     b_1 = \gamma^3 + \xi^{1/2}, \quad b_2 = \gamma^{3} - \xi^{1/2}.
    $$
    Thus, the corresponding Gaussian approximant is:
    \begin{align*}
    \widetilde{W_n}= \frac{2a_1}{\sqrt{n}}Z_1 + \frac{2a_2}{\sqrt{n}}Z_2  + \frac{b_1}{n} (Z_1^2 -1) + \frac{b_2}{n}(Z_2^2-1).
    \end{align*}
    As before, we examine the normalized statistic $n^{1/2}\widetilde{W_n}$, which is given by
    $$
    n^{1/2}\widetilde{W_n}= 2a_1 Z_1+ 2a_2Z_2  + \frac{b_1}{n^{1/2}} (Z_1^2 -1) + \frac{b_2}{n^{1/2}}(Z_2^2-1).
    $$
    Note that the statistic with normalization by $n$
      $$
    n\widetilde{W_n} = \frac{2a_1}{\sqrt{n}}Z_1 + \frac{2a_2}{\sqrt{n}}Z_2 + b_1(Z_1^2-1) + b_2(Z_2^2 - 1)
    $$
    cannot have a nontrivial distributional limit as $a_1,a_2$ are fixed nonzero constants. Thus, we do not consider its asymptotic behavior.
    \newline 
    \newline
    Moving to the notation of Theorem \ref{Thm:Full Bound}, we have:
    $$
    \mu_1 = 2a_1, \quad \mu_2 = 2a_2, \quad \l_1 = \frac{b_1}{n^{1/2}},\quad  \l_2 = \frac{b_2}{n^{1/2}}.
    $$
    We compute:
    $$
    \L_1^2 = \frac{1}{n}(b_1^2 + b_2^2) = \frac{1}{n}\sqbrace{\paren{\gamma^3 + \xi^{1/2}}^2 + \paren{\gamma^3 - \xi^{1/2}}^2} = \frac{2(\gamma^6 + \xi)}{n},
    $$
    and
    $$
    \norm{\mu}{2}^2 = \paren{(2a_1)^2 +(2a_2)^2} = 4.
    $$
    Finally:
    $$
    \frac{1}{p_1} = \frac{\l_1^2}{\L_1^2} = \frac{b_2^2}{2(\gamma^6 + \xi)} = \frac{(\gamma^3 + \xi^{1/2})^{2}}{2(\gamma^6+\xi)} = \frac{1}{2} + \frac{\gamma^{3}\xi^{1/2}}{2(\gamma^6+\xi)},
    $$
    $$
    \frac{1}{p_2} = \frac{\l_2^2}{\L_1^2} = \frac{b_2^2}{2(\gamma^6 + \xi)} = \frac{(\gamma^3 - \xi^{1/2})^2}{2(\gamma^6 + \xi)}  = \frac{1}{2} - \frac{\gamma^{3}\xi^{1/2}}{2(\gamma^6 + \xi)}.
    $$
    As we have exactly two eigenvalues, we must have that $p_1 = p_2 = 2$, or $p_1 < 2$ in which case $p_2 > 2$. Because the bounds obtained in these two cases are materially different, we divide the analysis into cases. Throughout, we assume that as $n \to \infty$, $p_1$ (and as a result $p_2$) remains fixed.
    \begin{description}
        \item[Case 1: ($\gamma = 0$, e.g. $p_1 = 2$)] 
    For the first case, this necessitates:
    $$
    \frac{1}{p_1} = \frac{1}{p_2} = \frac{1}{2} \implies \gamma^3 \xi^{1/2} = 0 \iff \gamma = 0 \text{ or } \xi = 0.
    $$
    Because $a_1,a_2$ are only well-defined when $\xi \neq 0$, we assume that $\gamma = 0$. We then obtain:
    $$
    b_1 = \xi^{1/2}, b_2 =-\xi^{1/2}, \quad a_1 = a_2 = \frac{1}{2^{1/2}}, \quad \L_1^2 = \frac{1}{n}(b_1^2 + b_2^2) = \frac{2\xi}{n}.
    $$
    Because $p_1 = 2$, we apply Theorem \ref{Thm:Full Bound} or Theorem \ref{Thm:User Friendly Bound} (as both give the same bound in this case) to obtain a logarithmic rate. More precisely, we obtain:
    $$
    Q_{n^{1/2}\widetilde{W_n}}(\e) \leq 35\e \sqbrace{\paren{\frac{1}{\frac{2\xi}{n} + 4}}^{1/2} + \mathbbm{1}\cbrace{{\e \leq 2^{3/2}\frac{\xi^{1/2}}{n^{1/2}}}}\exp\paren{-\frac{n}{2\xi}} \log\paren{\frac{2^{3/2}\xi^{1/2}}{n^{1/2}\e}}}.
    $$
    When we have that $\frac{\xi}{n} \to c^2 \in (0, +\infty)$, we obtain the limiting bound:
    $$
        Q_{n^{1/2}\widetilde{W_n}}(\e) \leq 35\e \sqbrace{\paren{\frac{1}{c^2 + 4}}^{1/2} + \mathbbm{1}\cbrace{\e \leq 2^{3/2} c}\exp\paren{-c^2} \log\paren{\frac{2^{3/2}c}\e}}.
    $$
    which corresponds to a mixed Gaussian - Chi square rate, owing to the linear combination of a linear and log-times-linear term. On the other hand, if the limit $c$ is $0$, then this bound reduces to:
    $$
    \frac{35\e}{2}
    $$
    which is exactly a Gaussian bound. 
    \item[Case 2: ($\gamma > 0$, e.g. $p_1 < 2$)] In this case, we have two possible bounds, namely the bound which can be obtained which can be obtained by application of Theorem 1, and the simpler bound which can be obtained by application of Theorem 2. For brevity, we simply state the bounds and asymptotics here, full computations can be found in our appendix. Borrowing the notation of \ref{Sec:Estimate} equation \eqref{eq:Gaussian-and-chi-square-anti-concentration}, we will only state estimates for the chi-square portion of the estimate, $\e Q_C(\e)$, as it is the leading order term in this case. 
    \begin{enumerate}
        \item \textbf{(Bounds obtained from Theorem \ref{Thm:Full Bound})}
        Plugging in the relevant variables into our bound from Theorem \ref{Thm:Full Bound}, we obtain an estimate for $Q_C(\e)$:
        \begin{equation}
            \frac{n^{1/2}}{[2(\gamma^6 + \xi)]^{1/2}}\paren{\frac{2(\gamma^6 + \xi)}{\gamma^{3}\xi^{1/2}}} \sqbrace{\paren{\frac{n\e^2}{4(\gamma^6 + \xi)}}^{{\frac{-\gamma^3\xi}{(\gamma^3 + \xi^{1/2})^2}}}-1}^{\frac{1}{2} + \frac{\gamma^3\xi^{1/2}}{2(\gamma^6 + \xi)}}
    \sqbrace{1 - \paren{\frac{n\e^2}{4(\gamma^6 + \xi)}}^{\frac{\gamma^3\xi}{(\gamma^3 + \xi^{1/2})^2}}}^{\frac{1}{2} - \frac{\gamma^3\xi^{1/2}}{2(\gamma^6 + \xi)}}
        \end{equation}
        Let us consider the case when:
    $$
    \frac{\gamma^6+\xi}{n} \to c^2, \quad c \neq 0
    $$
    This is equivalent to assuming that $\L_1$ converges to $|c|$, while the quantities $p_1,p_2$ remain fixed as $n \to \infty$, i.e. $n^{1/2}\widetilde{W_n}$ has a limit with a non-trivial chi-square part.. In such a case, we obtain Then, after fixing $2c >\e>0$ and passing to the limit, the above terms become:
    \begin{align}
    & \quad \frac{1}{c}\paren{\frac{2(\gamma^6 + \xi)}{\gamma^{3}\xi^{1/2}}} \sqbrace{\paren{\frac{\e^2}{4c^2}}^{{\frac{-\gamma^3\xi}{(\gamma^3 + \xi^{1/2})^2}}}-1}^{\frac{1}{2} + \frac{\gamma^3\xi^{1/2}}{2(\gamma^6 + \xi)}}
    \sqbrace{1 - \paren{\frac{\e^2}{4c^2}}^{\frac{\gamma^3\xi}{(\gamma^3 + \xi^{1/2})^2}}}^{\frac{1}{2} - \frac{\gamma^3\xi^{1/2}}{2(\gamma^6 + \xi)}} 
    \\&=\frac{1}{c}\paren{\frac{2(\gamma^6 + \xi)}{\gamma^{3}\xi^{1/2}}} \sqbrace{\paren{\frac{4c^2}{\e^2}}^{{\frac{\gamma^3\xi}{(\gamma^3 + \xi^{1/2})^2}}}-1}^{\frac{1}{2} + \frac{\gamma^3\xi^{1/2}}{2(\gamma^6 + \xi)}}
    \sqbrace{1 - \paren{\frac{\e^2}{4c^2}}^{\frac{\gamma^3\xi}{(\gamma^3 + \xi^{1/2})^2}}}^{\frac{1}{2} - \frac{\gamma^3\xi^{1/2}}{2(\gamma^6 + \xi)}}.
    \end{align}
    Examining this term more closely, we see that if $c$ is large relative to $\e$, the first term in square braces dominates, whereas if $c$ is small relative to $\e$, the coefficient $\frac{1}{c}$ is dominant. Furthermore, as $c \to +\infty$, we observe that the above term tends to $0$. This can be observed more generally by taking $\L_1 \to +\infty$ in the bound for $\mathcal{I}$ in Theorem $\ref{Thm:Full Bound}$. Alternatively, one may look at the computations of \ref{sec: t stat appendix analysis} to confirm that this is the case. 
    \item \textbf{(Bounds obtained from Theorem \ref{Thm:User Friendly Bound})}
    Applying Theorem \ref{Thm:User Friendly Bound} provides a simplified bound. However, as noted in the discussion of \ref{Sec:Estimate}, said bound does not recover the logarithmic rate as we send $\gamma \to 0^+$, which is one of its main deficiencies. However it is still useful to understand its behavior. making the relevant substitutions, we obtain an estimate of the form:
    \begin{equation}
            Q_C(\e) \leq \frac{(\gamma^6 + \xi)}{\gamma^3\xi^{1/2}}\paren{\frac{n^{1/2}}{2^{3/2}(\gamma^6 + \xi)^{1/2}}}^{\frac{3}{2} - \frac{\paren{\gamma^3 - \xi^{1/2}}^2}{2(\gamma^6 + \xi)}} \e^{-\frac{\gamma^3\xi^{1/2} - 2(\gamma^6 + \xi)}{2(\gamma^6 + \xi)}}.
    \end{equation}
    Note that taking $\gamma \to 0$ in this bound causes it to diverge. Thus, such a bound is best suited for cases when $\gamma$ is large, which in terms of eigenvalues means that the first eigenvalue dominates the remaining ones. As before, if we assume that:
    $$
    \frac{\gamma^6+\xi}{n} \to c^2, \quad c \neq 0
    $$
    with $p_1, p_2$ being fixed as $n \to \infty$, our estimate for $Q_C$ becomes:
    $$
     Q_C(\e) \leq \frac{(\gamma^6 + \xi)}{\gamma^3\xi^{1/2}}|c|^{\frac{3}{2} - \frac{\paren{\gamma^3 - \xi^{1/2}}^2}{2(\gamma^6 + \xi)}} \e^{-\frac{\gamma^3\xi^{1/2} - 2(\gamma^6 + \xi)}{2(\gamma^6 + \xi)}}
    $$
    For a further analysis, see Section \ref{sec: t stat appendix analysis}.
    \end{enumerate}

    \end{description}
    \section{Proof of Theorem 1}
\label{Sec: Proof}
\begin{proof}
The proof breaks naturally into several steps.
\begin{enumerate}
    \item{\textbf{The Convergence of and Characteristic Function of $S$}}
    We begin by observing that we may make sense of the random series
        \begin{equation*}
    S(\omega) = \sum_{k=1}^\infty \lambda_k(Z_k^2(\omega) - 1) + \mu_k Z_k(\omega)
    \end{equation*}
    in the pointwise (a.s.) sense, as it converges for $\bb P$ - a.s. $\omega$ by Kolmogorov's two-series theorem. Indeed:
    \begin{equation*}
    \E\sqbrace{\lambda_k(Z_k^2 - 1) + \mu_k Z_k} = 0
    \end{equation*}
    And:
    $$
    \mathrm{Var}(\lambda_k(Z_k^2 - 1) + \mu_kZ_k) = \E[(\l_kZ_k^2 + \mu_kZ_k)^2] - \lambda_k^2 = 2\lambda_k^2 + \mu_k^2
    $$
    which are both summable by our assumption. 
    \newline 
    \newline 
    Turning towards the estimation of $Q_S$, we apply Lemma 1.16 of \cite{Petrov}: for any real-valued random variable $X$ with characteristic function $\phi_X(t) := \bb E [\exp(itX)]$ we 
    $$
    Q_X(\e) \leq \paren{\frac{96}{95}}^2 \e \int_{-\frac{1}{\e}}^{\frac{1}{\e}}|\phi_X(t)| \d t \leq 2.2\e\int_{-\frac{1}{\e}}^{\frac{1}{\e}}|\phi_X(t)| \d t .
    $$
    In our case, we have with $X = S$ that:
    \begin{align}
    |\phi_{S}(t)| &= \left|\E\sqbrace{\exp\paren{it \sqbrace{\lim_{n \to \infty} \sum_{k=1}^n \lambda_{k}Z_k^2 + \mu_kZ_k}}}\right| \\ 
    &= \left|\E\sqbrace{\lim_{n \to \infty} \exp\paren{it \sum_{k=1}^n \lambda_{k}Z_k^2 + \mu_kZ_k}}\right| \label{eq:DCT}
    \\ 
    &= \lim_{n \to \infty}\left|\E\sqbrace{\exp\paren{it \sum_{k=1}^n \lambda_{k}Z_k^2 + \mu_kZ_k}}\right| \\ 
    &= \prod_{k=1}^{+\infty} |\phi_{\l_k(Z^2 -1) + \mu Z}(t)| \\ 
    &=  \frac{\exp\paren{-\sum_{k=1}^\infty\frac{2\mu_k^2t^2}{1+4\lambda_k^2t^2}}}{\prod_{k=1}^\infty(1+4\lambda_k^2t^2)^{1/4}}. \label{eq:Final CF}
    \end{align}
    Here, we used the dominated convergence theorem to interchange the limit and expectation, and the fact that the characteristic function of an independent sum is the product of those of the summands. Finally, in the last equality, we used that:
    $$
    |\phi_{\lambda (Z^2-1) + \mu Z}(t)| = \frac{\exp\paren{\frac{-2\mu^2t^2}{1+4\lambda^2t^2}}}{(1+4\lambda^2t^2)^{1/4}},
    $$
    as can be verified by a direct calculation.
    \newline 
    \newline 
    We thus must estimate:
    \begin{equation}
    \label{Eq:CF Integral}
    \int_{-\frac{1}{\e}}^{\frac{1}{\e}} |\phi_S(t)| \d t =\int_{-\frac{1}{\e}}^{\frac{1}{\e}}\frac{\exp\paren{-\sum_{k=1}^\infty\frac{2\mu_k^2t^2}{1+4\lambda_k^2t^2}}}{\prod_{k=1}^n(1+4\lambda_k^2t^2)^{1/4}} \d t = 2\int_{0}^{\frac{1}{\e}} \frac{\exp\paren{-\sum_{k=1}^\infty\frac{2\mu_k^2t^2}{1+4\lambda_k^2t^2}}}{\prod_{k=1}^n(1+4\lambda_k^2t^2)^{1/4}}\d t .
    \end{equation}
    \item{\textbf{Applying the Heuristics of Section (\ref{ProofTech})}}
    
To estimate \eqref{Eq:CF Integral}, we make precise the heuristics of small $t$, large $t$ which we discussed informally in Section \ref{ProofTech}. Set $t_0 := \frac{1}{2\L_1}$. Then for any $t \in [0, t_0]$ we have that:
    \begin{align*}
    |\phi_S(t)| = \exp\paren{-\sum_{k=1}^\infty\frac{2\mu_k^2t^2}{1 + 4\lambda_k^2t^2}} &\leq \exp\paren{-\sum_{k=1}^\infty\frac{2\mu_k^2t^2}{1 + 4\lambda_k^2t_0^2}} \\ 
    &= \exp\paren{-\L_1^2\sum_{k=1}^\infty \frac{2\mu_k^2t^2}{\L_1^2 + \lambda_k^2}} \\ 
    &\leq \exp\paren{- \L_1^2\sum_{k=1}^\infty \frac{2\mu_k^2 t^2}{2\L_1^2}} \\  
    &= \exp(-\norm{\mu}{2}^2t^2).\stepcounter{equation}\tag{\theequation}\label{Ineq: Gaussian Part Estimate}
\end{align*} 
 We  divide the interval of integration in \eqref{Eq:CF Integral} into two disjoint intervals, one being $\left[0, \frac{1}{2\L_1}\right]$ and the other being the interval $\left(\frac{1}{2\L_1}, \frac{1}{\e}\right]$. We will then apply the estimate (\ref{Ineq: Gaussian Part Estimate}) on the first interval and handle the second interval separately. 
 \newline 
 \newline 
 Before doing this, we note that it is possible for $\frac{1}{\e} \leq \frac{1}{2\L_1}$, so that partitioning the interval this way may not make sense. However, we always have the inequality:
\begin{equation}
\int_{0}^{\frac{1}{\e}} f(t) \d t \leq \int_{0}^{\frac{1}{2\L_1}} f(t)  \d t + \mathbbm{1}{\{\e \leq 2\L_1\}}\int_{\frac{1}{2\L_1}}^{1/\e } f(t) \d t.
    \end{equation}
    whenever $f$ is nonnegative and measurable. Thus, applying this to  \eqref{Eq:CF Integral} gives:
    \begin{equation}
        \int_{0}^{\frac{1}{\e}} |\phi_S(t)| \d t \leq \int_{0}^{\frac{1}{2\L_1}} |\phi_S(t)|  \d t + \mathbbm{1}{\{\e \leq 2\L_1\}}\int_{\frac{1}{2\L_1}}^{1/\e } |\phi_S(t)| \d t.
    \end{equation}
    Hereafter, we denote:
    $$
    \mathrm{I} = \int_{0}^{\frac{1}{2\L_1}}  |\phi_S(t)|\d t, \quad \mathrm{II} =  \int_{\frac{1}{2\L_1}}^{1/\e } |\phi_S(t)| \d t.
    $$
    Summarizing what we have done so far, we have shown:  
    \begin{equation}Q_S(\e) \leq 2.2 \e [2(\mathrm{I} + \mathbbm{1}{\{\e \leq 2\L_1\}}\mathrm{II})] = 4.4 \e(\mathrm{I} + \mathbbm{1}{\{\e \leq 2\L_1\}}\mathrm{II}). 
    \label{Eq: I, II AC Estimate}
    \end{equation}
    The remainder of the proof is the estimation of $\mathrm{I}$ and $\mathrm{II}$. 
    \item{\textbf{Estimation of $\mathrm{I}$ (The Gaussian Part):}}
    For $\mathrm{I}$, applying (\ref{Ineq: Gaussian Part Estimate}) and computing gives us:
    \begin{align*}
    \mathrm{I} 
    &\leq \int_{0}^{\frac{1}{2\L_1}}\exp\paren{-\norm{\mu}{2}^2t^2} \d t \\ 
    &= \int_{0}^{\frac{1}{2\L_1}} \exp\paren{-2\norm{\mu}{2}^2 \frac{t^2}{2}} \d t \\
    &= \frac{1}{2^{1/2}\norm{\mu}{2}}\int_0^{\frac{1}{\sqrt{2}} \frac{\norm{\mu}{2}}{\L_1}}  \exp\paren{-\frac{t^2}{2}} \d t \\
    &= \frac{\pi^{1/2}}{ \norm{\mu}{2}} \sqbrace{\Phi\paren{\frac{1}{2^{1/2}} \frac{\norm{\mu}{2}}{\L_1}} - \Phi(0)}.
    \end{align*}
     To simplify this further, we apply Corollary 2.2 of \cite{BobkovChistyakov} on the concentration functions of log-concave random variables. It states that for $X$ with log-concave density, we have:
     \begin{equation}
         Q_X( \e) \leq \frac{\e}{\sqrt{\mathrm{Var}(X) + \frac{\e
         ^2}{12}}}.
     \end{equation}
    In our case, setting $X = Z$, it gives:
    \begin{align*}
    \Phi\paren{\frac{1}{\sqrt{2}} \frac{\norm{\mu}{2}}{\L_1}} - \Phi(0) &= \bb P\paren{0 \leq Z \leq \frac{1}{\sqrt{2}} \frac{\norm{\mu}{2}}{\L_1}}  \\&\leq Q_Z\paren{\frac{1}{2^{1/2}} \frac{\norm{\mu}{2}}{\L_1}} \\ 
    &\leq  \frac{\frac{1}{2^{1/2}} \frac{\norm{\mu}{2}}{\L_1}}{\sqrt{1 + \frac{1}{24}\frac{\norm{\mu}{2}^2}{\L_1^2}}} \\ 
    &= \frac{\norm{\mu}{2}}{\paren{2\L_1^2 + \frac{\norm{\mu}{2^2}}{12}}^{1/2}}.
    \end{align*}
    The final bound we obtain for $\mathrm{I}$ is thus:
    \begin{equation}
    \frac{\pi^{1/2}}{ \norm{\mu}{2}} \sqbrace{\Phi\paren{\frac{1}{\sqrt{2}} \frac{\norm{\mu}{2}}{\L_1}} - \Phi(0)} \leq \frac{\pi^{1/2}}{ \norm{\mu}{2}} \frac{\norm{\mu}{2}}{\paren{2\L_1^2 + \frac{\norm{\mu}{2^2}}{12}}^{1/2}} = \paren{\frac{\pi}{2\L_1^2 + \frac{\norm{\mu}{2^2}}{12}}}^{1/2}.
    \label{Eq: Final Estimate of I}
    \end{equation}
    This completes the estimation of $\mathrm{I}$. 
    \item{\textbf{Estimation of II (The chi-square Part):}} 
    We only need to estimate $\mathrm{II}$ when $\e \leq 2\L_1$ as otherwise the indicator in front of it in \eqref{Eq: I, II AC Estimate} is equal to $0$. Thus, hereafter we assume $\e \leq 2\L_1$. 
    \newline
    \newline 
    Observe that the function \[ t \mapsto \frac{2\mu_k^2t^2}{1 + 4\lambda_k^2t^2}\] is increasing on $[0, + \infty)$ and hence the composition $t \mapsto \exp\paren{-\sum_{k=1}^\infty\frac{2\mu_k^2t^2}{1 + 4\lambda_k^2t^2}}$ is decreasing. For any $t$ with $\frac{1}{2\L_1} \leq t \leq \frac{1}{\e}$ we obtain:
    $$
    \exp\paren{-\sum_{k=1}^\infty\frac{\frac{2\mu_k^2}{4\L_1^2}}{1+ \frac{\lambda_k^2}{\L_1^2}}} = \exp\paren{-\sum_{k=1}^\infty\frac{\frac{1}{2}\mu_k^2}{\L_1^2+ \lambda_k^2}} \leq \exp\paren{-\frac{1}{4}\frac{\norm{\mu}{2}^2}{\L_1^2}}.
    $$
    Hence for $\frac{1}{2\L_1} \leq t \leq \frac{1}{\e}$:
    $$
    |\phi_S(t)| \leq \exp\paren{-\frac{1}{4}\frac{\norm{\mu}{2}^2}{\L_1^2}}\prod_{k=1}^\infty \frac{1}{(1 + 4 \l_k^2 t^2)^{1/4}}.
    $$
    So that:
    $$
    \mathrm{II} \leq \exp\paren{-\frac{1}{4}\frac{\norm{\mu}{2}^2}{\L_1^2}}    \int_{\frac{1}{2\L_1}}^\frac{1}{\e} \prod_{k=1}^n \frac{1}{(1 + 4 \l_k^2 t^2)^{1/4}}\d t .
    $$
    Applying Lemma \ref{Lem:LambdaIntEst} to the integral on the right hand side gives
    \begin{equation}   \mathrm{II} \leq 6\exp\paren{-\frac{1}{4}\frac{\norm{\mu}{2}^2}{\L_1^2}}   \mathcal{I}(\e, \lambda) \label{Eq: Final Estimate of II};
    \end{equation}
    which completes the estimation of $\mathrm{II}$.
    \item{\textbf{Putting Everything Together:}}
     Plugging in the bounds \eqref{Eq: Final Estimate of I} and \eqref{Eq: Final Estimate of II} respectively for $\mathrm{I}, \mathrm{II}$ into \eqref{Eq: I, II AC Estimate} and replacing different universal constants by their maxima gives 
    \begin{align*}
        Q_S(\e) &\leq 4.4 \e \sqbrace{\paren{\frac{\pi}{2\L_1^2 + \frac{\norm{\mu}{2^2}}{12}}}^{1/2} +  6\mathbbm{1}{\{\e \leq 2 \L_1\}} \exp\paren{-\frac{1}{4}\frac{\norm{\mu}{2}^2}{\L_1^2}}{\mathcal{I}(\e, \l)} } \\ 
        &\leq 4.4\e \sqbrace{\frac{7}{(24 \L_1^2 + \norm{\mu}{2}^2)^{1/2} } + 6\mathbbm{1}{\{\e \leq 2 \L_1\}} \exp\paren{-\frac{1}{4}\frac{\norm{\mu}{2}^2}{\L_1^2}}{6\mathcal{I}(\e, \l)} } \\ 
        &\leq 35 \e \sqbrace{\frac{1}{(\L_1^2 + \norm{\mu}{2}^2)^{1/2}} + \mathbbm{1}{\{\e \leq 2 \L_1\}}\exp\paren{-\frac{1}{4}\frac{\norm{\mu}{2}^2}{\L_1^2}} \mathcal{I}(\e, \l)}.
    \end{align*}
    which is the desired bound of Theorem \ref{Thm:Full Bound}.
    \end{enumerate}
    \end{proof}
    \section{Conclusion and Future Work} 
    \label{Sec: Conclusion}
    In this note, we obtained estimates for the concentration function of a Gaussian quadratic form, improving upon and generalizing the existing estimates in the literature. We also provided applications of our bounds to the limiting distributions of certain U-statistics. Continuing in this direction, several natural avenues for future work arise, which we describe below:
    \begin{enumerate}
        \item{\textbf{Simplified bounds and analysis}}:  
        Handling the chi-square portion of our bound seems to involve a large amount of casework, and it would be interesting to see if a more unified approach exists to obtain such estimates. In addition, our bounds do not recover optimal rates in a special case where $\l_k \geq 0$ for all $k\ge1$. It would be interesting to see if this can be taken into account without any cumbersome casework.
        \item \textbf{Berry-Esseen bounds for second order U-statistics}: Current proofs of Berry-Esseen Bounds for U-statistics use simpler, but ultimately, suboptimal bounds for the concentration function of quadratic forms. For example, \cite{huanggandy} proves an estimate for the Kolmogorov metric using the celebrated Carbery-Wright inequality. Although powerful and general, it suffers from the drawback that it always provides a rate of $\e^{1/2}$ for quadratic forms. \cite{yanus} provides rates for degenerate U-statistics by estimating the concentration function of the sum with simply that of the largest term. Thus, to our knowledge, no work has used a concentration function estimate with a rate other than $\e^{1/2}$, and we plan to explore how our new estimate can provide improved convergence rates in the Kolmogorov metric. 
        \item {\textbf{Extension to higher order U-statistics}}: Our current estimates are for Gaussian quadratic forms, and these are precisely the limiting distributions for second order U-statistics. In general, for an $m$-th order $U$-statistic, the appropriate limit distribution is a Gaussian polynomial of order $m$. Such random variables do not admit closed-form expressions for the characteristic function; making it difficult for us to derive such estimates directly. However, in this regard, the technique of \textit{symmetrization}, introduced in \cite{gotze1979} for quadratic forms in independent random variables and later extended by (for example) \cite{yurinskii1979} for higher order polynomials, can be helpful in estimating such quantities. In particular, when estimating the characteristic function, such results allow us to replace a nonlinear expression in the Gaussian random variables by a linear one, after appropriately symmetrizing. Thus, a further line of work could consist of proving analogous estimates for the concentration function to what we have shown here using Fourier-Analytic techniques, and subsequently using these to obtain sharper Berry-Esseen bounds for higher order $U$-statistics.
        \item{\textbf{A Complete Set of Lower Bounds}}: All of our estimates are currently upper bounds, and in certain cases, we obtain matching lower bounds. For example, in the case ${1}/{p_1} \leq {1}/{3}$, we obtain rates that match those of Theorem $2$ of \cite{bobkov2020two}, which are in turn two sided bounds. Our computations in section \ref{sec:concentration function lower bound} also show that in the special case of $\l_1= 1, \l_2 = -1$ and $\mu = 0$, we obtain a matching lower bound for $\e$ taken sufficiently small. To obtain lower bounds in all regimes of $\mu, \l$ would be an interesting line of future work.
    \end{enumerate}
        \bibliography{references}
\bibliographystyle{apalike}
\newpage
\setcounter{section}{0}
\setcounter{equation}{0}
\setcounter{figure}{0}
\renewcommand{\thesection}{S.\arabic{section}}
\renewcommand{\theequation}{E.\arabic{equation}}
\renewcommand{\thefigure}{A.\arabic{figure}}

\begin{center}
  \Large {\bf Appendix to ``On the L{\'e}vy concentration function of Gaussian quadratic forms with applications to second order U-statistics"}
  \end{center}
\section{The Kolmogorov-Rogozin Inequality} 
\label{Sec: KR Ineq}
Our initial approach for obtaining such estimates was to use the Kolmogorov-Rogozin inequality. It bounds the concentration function of a sum of independent random variables in terms of the concentration function of the summands. A version given in \cite{EsseenKolmRog} states that for $\{X_k\}_{k\in \mathbb{N}}$ independent and $S = \sum_{k=1}^{+\infty} X_k$ (the series interpreted $\mathbb{P}$ a.s.), we have that:
    \begin{align*}
    \label{Eq: Kolmogorov Rogozin}
    Q_{S}(\e) \lesssim \e \paren{\sum_{k=1}^{\infty} \inf_{0\leq \e_k \leq \e}\e_k^2D(\widetilde{X_k},\e_k)}  &\lesssim \frac{\e}{\sqrt{\sum_{k=1}^{\infty} \sup_{0\leq \e_k \leq \e} \e_k^2(1-Q_{\widetilde{X_k}}(\e))}} \\ 
    &\leq \frac{\e}{\sqrt{\sum_{k=1}^{\infty} \sup_{0\leq \e_k \leq \e} \e_k^2(1-Q_{{{X_k}}}(\e))}};
    \end{align*}
    where here $D$ denotes the following truncated second moment: 
    $$
    D(X,c) = \E\sqbrace{|X|^2\mathbbm{1}{\{|X|\leq c\}}} + c^2\bb P(|X| > c),
    $$
    and $\widetilde{X}$ is a symmetrized version of $X$, e.g. $\widetilde{X} = X - X'$ where $X'$ is an independent copy of $X$. In general, the easiest estimates to apply are the two rightmost, which require only knowledge of $Q_{\widetilde{X_k}}$ or $Q_{X_k}$ for individual random variables $X_k$. We always have that $Q_{\widetilde{X}} \leq Q_{X}$, which gives us the second inequality above. However, sometimes $Q_{\widetilde{X}}$ may be difficult to estimate compared to $Q_X$ and so the second inequality may be used. Prior to applying the inequality, let us also note that if we have an upper bound for the concentration function, e.g.:
    $
    Q_{X_k}(\e) \leq M_k(\e)
    $
    then the inequality holds with $M_k$ in place of $Q_{X_k}$, that is:
    $$
    Q_{S}(\e) \leq \frac{\e}{\sqrt{\sum_{k=1}^{\infty} \sup_{0\leq \e_k \leq \e} \e_k^2(1-M_{k}(\e))}};
    $$
    which is how the inequality would be applied in practice, as the exact form $Q_{X_K}(\e)$ is almost never known but can usually be bounded from above without difficulty. 
    \newline 
    \newline 
    To demonstrate the deficiency of the inequality for our applications, we show it yields poor bounds in the case of  sums of central chi-square random variables, i.e. we consider:
    $$
    S = \sum_{k=1}^{\infty} \lambda_k(Z_k^2 - 1).
    $$
    Although matching two sided bounds have already been proven in this case by \cite{bobkov2020two}, it is instructive and straightforward to apply the inequality in this case as $Q_{X_k}(\e)$ has a simple upper bound. Indeed,  $X_k = \l_k(Z_k^2 - 1)$, and by translation invariance, $Q_{X_k}(\e) = Q_{\l_kZ_k^2}(\e) = Q_{\Z_k^2}\paren{\frac{\e}{|\l_k|}} $. For any $x \in \mathbb{R}$ we have:
    \begin{align*}
    \bb{P}\paren{x \leq Z_k^2 < x + \frac{\e}{|\l_k
    |
    }} &= \int_{x}^{x + \frac{\e}{|\l_k|}} \frac{1}{z^{1/2}}\exp(-z/2) \mathbf{1}\cbrace{x \geq 0} \d z  \\ 
    &\leq \int_0^{\frac{\e}{|\l_k|}} \frac{1}{z^{1/2}}\exp(-z/2) \d z \\ 
    &\leq 2 \paren{\frac{\e^{1/2}}{|\l_k|^{1/2}}}.
    \end{align*}
    Hence, we have the estimate:
    $$
    Q_{X_k}(\e) \leq \min \cbrace{1, 2 \frac{\e^{1/2}}{|\l_k|^{1/2}}}.
    $$
     Thus, for each $k$ we must compute:
    $$
    \sup_{0 \leq \e_k \leq \e} \e_k^2\paren{1 - \min \cbrace{1,2\frac{\e_k^{1/2}}{|\l_k|^{1/2}}}}.
    $$
    For the bound to not trivialize (e.g. for the supremum to be nonzero), we would need for all $k$ that:
    $$
    2\frac{\e_k^{1/2}}{|\l_k|^{1/2}} < 1 \implies \e_k < \frac{|\l_k|}{4}.
    $$
    Note that we also have the constraint $\e_k \leq \e$.
    Thus, we must compute:
    $$ 
    \sup_{0 \leq \e_k \leq \min \cbrace{\e, \frac{|\l_k|}{4}}} \e_k^2\left(1 - 2\frac{\e_k^{1/2}}{|\l_k|^{1/2}}\right).
    $$
    We consider the function $x^2\paren{1 - 2\frac{x^{1/2}}{|\l_k|^{1/2}}}$. The function is $0$ for $x>0$ and positive for $x$ small, so any maxima be identified by computing critical points. We differentiate and find the critical point:
    $$
    x^* = \frac{4}{25}|\l_k|.
    $$
    Note that $\frac{4}{25}|\l_k| < \frac{|\l_k|}{4}$, so that all that remains to be seen is whether $\frac{4}{25}|\l_k| < \e$. Assuming that $ \e > \frac{4}{25} |\l_k|$, we obtain the value:
    $$
    \paren{\frac{4}{25}|\l_k|}^2\frac{1}{5} = \frac{16}{3125}|\l_k|^2.
    $$
    On the other hand, if $\e$ is smaller than $\frac{4}{25}|\l_k|$ this value, then the maximum is achieved at $\e$ as this is an increasing function to the left of the critical point. Thus, we obtain:
    $$
    \sup_{0 \leq \e_k \leq \min \cbrace{\e, \frac{|\l_k|}{4}}} \e_k^2\left(1 - 2\frac{\e_k^{1/2}}{|\l_k|^{1/2}}\right) = \begin{cases}
         \frac{16}{3125}|\l_k^2| & \e >\frac{4}{25}|\l_k|, \\ 
         \e^2\paren{1 -2\frac{\e^{1/2}}{|\l_k|^{1/2}} } & \e \leq \frac{4}{25}|\l_k|.
    \end{cases}
    $$
    The sequence $|\l_k| \to 0$ and is monotone, so there exists some $K_\e$ for which $\e \leq \frac{4}{25}|\l_k|$ for $k \leq K_\e$ and $\e > \frac{4}{25}|\l_k|$ for $k > K_\e$. Hence, we obtain a bound of the form:
    $$
    \frac{\e}{\sqrt{\sum_{k=1}^{K_\e}  \e^2\paren{1 -2\frac{\e^{1/2}}{|\l_k|^{1/2}} } + \frac{16}{3125}\L_{K_\e + 1}^2}}.
    $$
    To see the issues with this bound, we consider the case when presented a sum with finitely many terms, say $N$ terms so that $\l_{k} = 0$ for all $k > N+1$ and $\l_k \neq 0$ for $1 \leq k \leq N$. Then, for any $\e \leq \frac{4}{25}|\l_N|$ we obtain the lower bound:
    $$
    \frac{\e}{\sqrt{\e^2\sum_{k=1}^N\paren{1 - 2 \frac{\e^{1/2}}{|\l_k|^{1/2}}}}} \geq \frac{1}{\sqrt{N}}.
    $$
    This is simply because $1 - 2\frac{e^{1/2}}{|\l_k|^{1/2}} \leq 1$ for this value of $\e$. Thus, in these cases, the Kolmogorov-Rogozin inequality provides the correct rate in $N$, the number of summands, but does not provide the correct rate in terms of $\e$ as $\e \to 0^+$ for fixed $N$. Indeed, all terms in the sum are absolutely continuous random variables and so as $\e \to 0^+ $ we should also have $Q_{S}(\e) \to 0$. However, the Kolmogorov Rogozin inequality gives us a bound which is constant order as $\e \to 0^+$. 
\section{The Carbery-Wright Inequality}
\label{Sec: CR Ineq}
Theorem 8 of \cite{CarberyWright} is a concentration function estimate for polynomials of a log-concave random variable. Specializing the result in our case to Gaussian quadratic forms, e.g. as is done in \cite{huanggandy}, it states that 
\begin{equation}
    Q_{S}(\e) \leq  \frac{C \e^{1/2}}{\mathrm{Var}(S)^{1/4}}
\end{equation}
for some universal constant $C > 0$. 
The Carbery-Wright inequality is powerful due to its impressive generality and ease of use -- it provides a bound for polynomials of \textit{any} degree and with \textit{any} log-concave random variable as input. This generality can also lead to suboptimal bounds in simple cases. In our setting, we have that 
\[
\mbox{Var}(S) = \sum_{k=1}^{\infty} \left[\lambda_k^2\mathbb{E}[(Z_k^2 - 1)^2] + \mu_k^2\mathbb{E}[Z_k^2]\right] = \sum_{k=1}^{\infty} (\mu_k^2 + 3\lambda_k^2) = \|\mu\|_2^2 + 3\|\lambda\|_2^2.
\]
Hence, recalling the notation of~\eqref{eq:Gaussian-and-chi-square-anti-concentration} this gives a bound of the form $Q_S(\e) \le CQ_G^{1/2}(\e)$. In particular, such a bound always gives a rate of $\e^{1/2}$ when it is possible in many cases to obtain the faster rate of $\e$. 

    \section{Density and Concentration Function of a Symmetrized chi-square Random Variable}
    \label{Pf: Symmetrized Concentration}
    Here, we compute the density of $Z_1^2 - Z_2^2$ where $Z_1, Z_2$ are independent, $\Normal(0,1)$ random variables. Although this is a straightforward computation, we provide it here for future reference as we could not find it elsewhere. We also provide a lower bound for the concentration function, showing that it is of order $\e \log (1/\e)$ for $\e$ sufficiently small. This provides evidence for the optimality of our bounds in Theorems \ref{Thm:Full Bound} and \ref{Thm:User Friendly Bound} in the $p_1 = 2$ case. 
    \subsection{Density Computation}
    We have:
    $$
    f_{Z^2} = \frac{1}{(2\pi)^{1/2}}\frac{1}{x^{1/2}}\exp(-x/2)\mathbf{1}_{x > 0}.
    $$
    Thus:
    \begin{align*}
    f_{Z_1^2 -Z_2^2}(x) &= \int_{\mathbb{R}}f_{Z^2}(x-z)f_{Z^2}(-z) dz  \\ 
    &= \frac{1}{2 \pi} \int_{-\infty}^{\min(0,x)} \frac{1}{(x-z)^{1/2}(-z)^{1/2}} \exp(-(x-z)/2)\exp(z/2) \d z \\ 
    &= \frac{\exp(-x/2)}{2 \pi} \int_{-\infty}^{\min(0,x)} \frac{1}{(x-z)^{1/2}(-z)^{1/2}} \exp(z)  \d z .
    \end{align*}
    As this density must be symmetric, without loss of generality we take $x > 0$ so that the $\min(0,x) = 0$. We change variables $z = xu$ and then $v = -u$ to  obtain:
    \begin{align*}
        \frac{\exp(-x/2)}{2 \pi} \int_{-\infty}^{0} \frac{1}{(x-z)^{1/2}(-z)^{1/2}} \exp(z)  \d z  
        &= \frac{\exp(-x/2)}{2 \pi} \int_{-\infty}^{0} \frac{1}{(1-u)^{1/2}u^{1/2}} \exp(xu)  \d u \\ 
        &= \frac{\exp(-x/2)}{2 \pi} \int_{0}^{+\infty} \frac{1}{(1+v)^{1/2}v^{1/2}} \exp(-xv)  \d v.
    \end{align*}
    This integral is finite for all non-zero $x$. However, as $x\to 0^+$, the monotone convergence theorem gives:
    $$
    \lim_{x \to 0^+} f_{Z_1^2 -Z_2^2}(x) = \frac{1}{2\pi}\int_{0}^{+\infty} \frac{1}{(1+u)^{1/2}u^{1/2}} \d u  = + \infty.
    $$
    Showing that the density is unbounded and that the concentration function cannot behave linearly as $\e \to 0^+$.
    \subsection{Concentration Function Lower Bound}
    \label{sec:concentration function lower bound}
    With this formula in hand, we provide a lower bound for the concentration function by lower bounding $\bb{P}(0 < Z_1^2 - Z_2^2 \leq \e)$. More concretely, we will show that for $\e < \exp\paren{-\frac{\exp(1)}{2}}$:
    $$
    \bb{P}(0 < Z_1^2 - Z_2^2 \leq \e) \geq \frac{1}{4\pi} \e \log \paren{\frac{1}{\e}}.
    $$
    For $\e > 0$, we compute:
    \begin{align*}
        (2\pi) \bb{P}(0 < Z_1^2 - Z_2^2 \leq \e) &= \int_0^\e\sqbrace{\exp\paren{-x/2}\int_{0}^{+\infty} \frac{\exp(-xv)}{(1+v)^{1/2}v^{1/2}}\d v} \d x \\ 
        &= \int_0^{+\infty}\sqbrace{\frac{1}{(1+v)^{1/2}v^{1/2}}\int_{0}^{\e} \exp\paren{-x/2 - xv}\d x }\d v
        \\&=\int_0^{+\infty}\frac{1}{(v^2+v)^{1/2}}\frac{1 - \exp\paren{\paren{-v + \frac{1}{2}}\e}}{v + \frac{1}{2}}\d v.
    \end{align*}
    with the interchange of integrals being justified by Tonelli's theorem. Completing the square in the denominator gives that the above is:
       \begin{align*}
        \int_0^{+\infty}\frac{1}{(v^2+v)}\frac{1 - \exp\paren{\paren{-v + \frac{1}{2}}\e}}{v + \frac{1}{2}} \d v &= \int_0^{+\infty}\frac{1}{\sqbrace{\paren{v + \frac{1}{2}}^2 - \frac{1}{4}}^{1/2}}\frac{1 - \exp\paren{\paren{-v + \frac{1}{2}}\e}}{v + \frac{1}{2}} \d v \\ 
        &= \int_{1/2}^{+\infty}\frac{1 - \exp\paren{-u\e}}{u\sqbrace{u^2 - \frac{1}{4}}^{1/2}} \d u \\
        &= \int_{\frac{\e}{2}}^{+\infty} \frac{1-\exp(-z)}{z\sqbrace{\frac{z^2}{\e^2} - \frac{1}{4}}^{1/2}} \d z.
    \end{align*}
    By Taylor's theorem, for $z$ in the interval $[0,c]$, we have the estimate:
    $$
    \exp(-z) \leq 1 - z + \exp(c) \frac{z^2}{2}.
    $$
    Take $c$ such that $\frac{c}{\e} > 1/2$ e.g. $c > \e/2$.
    We may truncate the integral to obtain:
    \begin{align*}
        \int_{\frac{\e}{2}}^{+\infty} \frac{1-\exp(-z)}{z\sqbrace{\frac{z^2}{\e^2} - \frac{1}{4}}^{1/2}} \d z \geq \int_{\frac{\e}{2}}^{c} \frac{z - \exp(c) 
        \frac{z^2}{2}}{z\sqbrace{\frac{z^2}{\e^2} - \frac{1}{4}}^{1/2}} \d z = \int_{\frac{\e}{2}}^{c} \frac{1}{\sqbrace{\frac{z^2}{\e^2} - \frac{1}{4}}^{1/2}} \d z  - \exp(c)\int_{\frac{\e}{2}}^{c} \frac{
        \frac{z^2}{2}}{z\sqbrace{\frac{z^2}{\e^2} - \frac{1}{4}}^{1/2}} \d z.
    \end{align*}
    For the first term:
    $$
    \int_{\frac{\e}{2}}^{c} \frac{1}{\sqbrace{\frac{z^2}{\e^2} - \frac{1}{4}}^{1/2}} \d z = \e \int_{1/2}^{c/\e} \frac{1}{\paren{u^2 - \frac{1}{4}}^{1/2}} \d u = \e\log\paren{2\sqbrace{\frac{c}{\e} + \paren{\frac{c^2}{\e^2} - \frac{1}{4}}^{1/2}}}. 
    $$
    For the second term, we compute:
    \begin{align*}
    \int_{\frac{\e}{2}}^{c} \frac{
        \frac{z^2}{2}}{z\sqbrace{\frac{z^2}{\e^2} - \frac{1}{4}}^{1/2}} \d z &= \int_{\frac{\e}{2}}^{c} \frac{z}{2\sqbrace{\frac{z^2}{\e^2} - \frac{1}{4}}^{1/2}} \d z \\&= \e^2 \int_{1/2}^{c/\e} \frac{u}{2\paren{u^{2} - \frac{1}{4}}^{1/2}} \d u
        \\
        &=\e^2 \int_{1/4}^{c^2/\e^2} \frac{1}{4\paren{z - \frac{1}{4}}^{1/2}} \d u 
        \\
        &= \frac{\e^2}{4} \sqbrace{\paren{\frac{c^2}{\e^2} - \frac{1}{4}}^{1/2}}.
    \end{align*}
    Adding the results, we obtain the estimate:
    $$
    \e\log\paren{2\sqbrace{\frac{c}{\e} + \paren{\frac{c^2}{\e^2} - \frac{1}{4}}^{1/2}}}  - \exp(c) \frac{\e^2}{4} \sqbrace{\paren{\frac{c^2}{\e^2} - \frac{1}{4}}^{1/2}}.
    $$
    To simplify this estimate a bit, we have:
    $$
    \e\log\paren{2\sqbrace{\frac{c}{\e} + \paren{\frac{c^2}{\e^2} - \frac{1}{4}}^{1/2}}} \geq \e \log\paren{\frac{c}{\e}},
    $$
    and:
    $$
    \frac{\e^2}{4} \sqbrace{\paren{\frac{c^2}{\e^2} - \frac{1}{4}}^{1/2}} \leq c\frac{\e}{4}.
    $$
    Assume without loss of generality that $\e < 1$. Then, we may take $c = 1$ (for example) and obtain the estimate:
    \begin{align*}
        \e\log\paren{2\sqbrace{\frac{c}{\e} + \paren{\frac{c^2}{\e^2} - \frac{1}{4}}^{1/2}}}  - \exp(c) \frac{\e^2}{4} \sqbrace{\paren{\frac{c^2}{\e^2} - \frac{1}{4}}^{1/2}}
 &\geq \e \log\paren{\frac{1}{\e}} - \e\frac{\exp(1)}{4} \\ 
 &= \e \sqbrace{\log\paren{\frac{1}{\e}} - \frac{\exp(1)}{4}}.
    \end{align*}
    Hereafter, we restrict to $\e$ such that:
    $$
    \e < \exp\paren{-\frac{\exp(1)}{2}}.
    $$
    For such $\e$, we have:
    $$
    \log\paren{\frac{1}{\e}} - \frac{\exp(1)}{4} > \frac{1}{2}\log\paren{`\frac{1}{\e}}.
    $$
    Hence:  
    $$
    \e \sqbrace{\log\paren{\frac{1}{\e}} - \frac{\exp(1)}{4}} \geq \frac{1}{2}\sqbrace{\e \log \paren{\frac{1}{\e}}}.
    $$
    Dividing both sides by the normalizing constant $2\pi$ provides the desired estimate. 
\section{Recovering the Bounds of Theorem 2 of \cite{bobkov2020two}}
\label{Sec: Recovering Bobkov}
In this section, we show that our bounds imply the concentration function estimate obtained from Theorem 2 of \cite{bobkov2020two}, up to universal constant factors. As we recover their estimate under weaker assumptions, we can consider our result a strengthening of theirs, modulo worse constants. 
\begin{proof}
    Strictly speaking, to apply Theorem 2 of \cite{bobkov2020two} exactly as stated, we must assume that we have $\l_1^2 \leq \frac{\L_1^2}{3}$, i.e. $\frac{1}{p_1} \leq \frac{1}{3}$ or $p_1 \geq 3$ and that $\l \geq 0$. Furthermore, their result applies only to finite sums. However, we emphasize that these assumptions are not necessary to apply our result.
    \newline 
    \newline 
    To apply their estimate to our case, we use their notation and set $a_k = -\frac{1}{2}\frac{\mu_k}{\l_k}$. If we define:
    \begin{equation}
    W_a := \sum_{k=1}^{n}\l_k(Z_k - a_k)^2,
    \end{equation}
    then we have:
    \begin{align*}
    W_a &= \sum_{k=1}^{n} \l_k Z_k^2 +\mu_kZ_k + \frac{1}{4}\frac{\mu_k^2}{\l_k^2} \\ 
    &= \sum_{k=1}^{n}\l_k(Z_k^2 -1) + \mu_kZ_k + \sum_{k=1}^{n}\l_k + \frac{1}{4} \frac{\mu_k^2}{\l_k^2} \\ 
    &= S +\sum_{k=1}^{n}\l_k + \frac{1}{4} \frac{\mu_k^2}{\l_k^2}. 
    \end{align*}
    Theorem $2$ of \cite{bobkov2020two} states that:
    \begin{equation}
    \norm{f_{W_a}}{L^\infty(\bb{R})} \leq \frac{2}{\paren{\L_1^2 +\sum_{k=1}^n \l_k^2 a_k^2}^{1/2}}.
    \end{equation}
    Because $W_a = S + C$ where $C$ is a constant, necessarily $\norm{f_Y}{L^\infty(\bb{R})} = \norm{f_{W_a}}{L^\infty(\bb{R})}$. Applying this fact gives an estimate
    $$
    Q_{S}(\e) \leq \frac{\e}{\paren{\L_1^2 +\sum_{k=1}^n \l_k^2 a_k^2}^{1/2}}.
    $$
    Substituting for $a_k$ gives:
    \begin{equation}
    \label{Eqn: Bobkov Simplified}
        Q_{S}(\e) \leq \frac{\e}{\paren{\L_1^2 + \frac{1}{4}\sum_{k=1}^n\mu_k^2 }^{1/2}} = \frac{2\e}{\paren{4 \L_1^2 + \norm{\mu}{2}^2}^{1/2}} := B(\e).
    \end{equation}
    Let $A(\e)$ denote our bound from Theorem (\ref{Thm:User Friendly Bound}). To show that our bound recovers theirs, we must show that there exists a universal constant $m > 0$ for which:
    $$
    B(\e) \geq m A(\e).
    $$
    First, we observe:
    $$
    B(\e) \geq \frac{2\e}{\paren{4\L_1^2 + 4\norm{\mu}{2}^2}^{1/2}} = \frac{\e}{\paren{\L_1^2 + \norm{\mu}{2}^2}^{1/2}}.
    $$
    Recalling Theorem \ref{Thm:User Friendly Bound}, our bound is of the form:
    \begin{equation}
        \label{Eqn: Our Bound}
        35 \e \sqbrace{ \paren{\frac{1}{\L_1^2 + \norm{\mu}{2}^2}}^{1/2}  + \mathbbm{1}{\{\e \leq 2 \L_1\}}\exp\paren{-\frac{1}{4}\frac{\norm{\mu}{2}^2}{\L_1^2}}\mathcal{J}(\e, \lambda)}.
    \end{equation}
    The first term in square brackets is exactly in the form of the bound given in \eqref{Eqn: Bobkov Simplified}. For the second term, $\exp\paren{-\frac{1}{4}\frac{\norm{\mu}{2}^2}{\L_1^2}} \mathcal{J}(\e,\l)$, we have $\frac{1}{p_1}\leq \frac{1}{3}$, which means that we may apply the bound corresponding to $p_1 > 2$ and obtain: 
    \begin{equation}
     \mathcal{J}(\e, \l) \leq \frac{1}{\frac{1}{2} - \frac{1}{p_1}}\frac{1}{\L_1} \leq \frac{6}{\L_1} .
    \end{equation}
    Thus, the second term is bounded by:
    \begin{equation}
    \label{Eqn: Simplified J Exp Bound}
    6 \cdot \frac{\exp\paren{-\frac{1}{4}\frac{\norm{\mu}{2}^2}{\L_1^2}}}{\L_1}.
    \end{equation}
    To complete the proof, we will show that there exists a universal constant $C> 0$ for which:
    $$
    \frac{\exp\paren{-\frac{1}{4}\frac{\norm{\mu}{2}^2}{\L_1^2}}}{\L_1} \leq \frac{C}{\paren{\L_1^2 + \norm{\mu}{2}^2}^{1/2}}.
    $$
    By nonnegativity of our expressions, we have the equivalence:
    \begin{align*}
         \frac{\exp\paren{-\frac{1}{4}\frac{\norm{\mu}{2}^2}{\L_1^2}}}{\L_1} \leq \frac{C}{\sqrt{\L_1^2 + \norm{\mu}{2}^2}} &\iff \paren{\L_1^2 + \norm{\mu}{2}^2}\frac{\exp\paren{-\frac{1}{2}\frac{\norm{\mu}{2}^2}{\L_1^2}}}{\L_1^2} \leq C^2 .
         \end{align*}
         Simplifying gives that $C$ must satisfy:
         \begin{align*}
         \exp\paren{-\frac{1}{2}\frac{\norm{\mu}{2}^2}{\L_1^2}} + \frac{\norm{\mu}{2}^2}{\L_1^2} \exp\paren{-\frac{1}{2}\frac{\norm{\mu}{2}^2}{\L_1^2}} \leq C^2.
    \end{align*}
    The first term in the sum on the left hand side is bounded by $1$, being the exponential of a negative number. For the second term, the function $x \mapsto x\exp(-x/2)$ achieves its maximum value $2 \exp(-1) \leq 1$ at $x = 2$. Thus, the above bound holds for $C^2 = 1 + 2\exp(-1) < 2$. Substituting this into \eqref{Eqn: Simplified J Exp Bound} gives:
    \begin{equation}
        \exp\paren{-\frac{1}{4}\frac{\norm{\mu}{2}^2}{\L_1^2}}\mathcal{J}(\e, \l) \leq \frac{6 \cdot 2^{1/2}}{\sqrt{\L_1^2 + \norm{\mu}{2}^2}}.
    \end{equation}
    Totaling everything in \eqref{Eqn: Our Bound} gives:
    \begin{equation}
        A(\e) \leq  \frac{(35 \cdot 6\cdot 2^{1/2})\e}{\paren{\L_1^2 + \norm{\mu}{2}^2}^{1/2}} \leq (35 \cdot 6\cdot 2^{1/2})B(\e).
    \end{equation}
    Thus, we obtain the desired estimate with $m = \frac{1}{(35 \cdot 6 \cdot 2^{1/2})}$.
\end{proof}
    \section{Proof of Lemma 1} \label{prf:LambdaIntEst}
        \begin{proof}
    We wish to estimate the integral:
    \begin{equation}
    \label{eq:initial integral}
    \int_{\frac{1}{2\L_1}}^\frac{1}{\e} \prod_{k=1}^\infty\frac{1}{(1 + 4 \l_k^2 t^2)^{1/4}}\d t .
    \end{equation}
    As defined in Section \ref{sec: Notation}, we set:
    $$
    \frac{1}{p_k} := \frac{\lambda_k^2}{\L_1^2}, \quad k \geq 1.
    $$
    Each $1/p_k$ describes the total contribution of the $\l_k$ term to the total variance of the sum. We have the following easy to verify properties:
    \begin{enumerate}
    \item $\sum_{k} \frac{1}{p_k} = 1,$
    \item $\frac{1}{p_k} \geq \frac{1}{p_{k+1}}$ or $p_{k} \leq p_{k+1}$ (by the fact that $|\l_k|$ are decreasing),
    \item $p_k \geq k$ for all $k$.
    \end{enumerate}
H\"older's inequality with exponents $p_k$ gives us that:
    \[
\int_{\frac{1}{2\L_1}}^\frac{1}{\e} \prod_{k=1}^\infty \frac{1}{(1 + 4 \l_k^2 t^2)^{1/4}}\d t
    \leq \prod_{k=1}^{\infty}     \paren{\int_{\frac{1}{2\L_1}}^\frac{1}{\e} \frac{1}{(1 + 4 \l_k^2 t^2)^{p_k/4}}\d t}^{1/p_k}.
    \]
    For the integral inside parentheses, we make the substitution:
\begin{equation} \label{eq: int substitution}
        u = 1 + 4 \lambda_k^2t^2, \implies t = \frac{(u - 1)^{1/2}}{2|\lambda_k|}, \d t = \frac{1}{4|\lambda_k|(u-1)^{1/2}} \d u.
        \end{equation}
        Hence, the integral becomes:
        \begin{equation}
        \label{Eq: I_k Integral}
        \int_{\frac{1}{2\L_1}}^\frac{1}{\e} \frac{1}{(1 + 4 \l_k^2 t^2)^{p_k/4}} \d t = \frac{1}{4|\lambda_k|}\int_{1 + \frac{\lambda_k^2}{\L_1^2}}^{1 + \frac{4\lambda_k^2}{\e ^2}} u^{-p_k/4}(u-1)^{-1/2}\d u = \frac{1}{4|\lambda_k|}\int_{1 + \frac{1}{p_k}}^{1 + \frac{4\lambda_k^2}{\e ^2}} u^{-p_k/4}(u-1)^{-1/2}\d u.
        \end{equation}
        We next search for an estimate of the kind:
        \begin{equation}\label{eq: poly sub}
        (u-1)^{-1/2} \leq C u^{-1/2}, \quad u \in \sqbrace{1 + \frac{1}{p_k}, 1 + \frac{4\lambda_k^2}{\e ^2}},
        \end{equation}
        so that \eqref{Eq: I_k Integral} can be bounded by the integral of $u^{-p_k/4 -1/2}$ times a constant. We observe that the function
        $
        \frac{u}{u-1}
        $
        is decreasing for $u \geq 1$. Hence, on our interval of integration, we must have:
        $$
                \frac{u}{u-1} \leq \frac{1 + \frac{1}{p_k}}{\frac{1}{p_k}} \implies (u-1)^{-1/2}\leq \paren{\frac{1 + \frac{1}{p_k}}{\frac{1}{p_k}} }^{1/2} u^{-1/2} = (p_k+1)^{1/2}u^{-1/2}.
        $$
        Hence:
        \begin{equation}
        \frac{1}{4|\lambda_k|}\int_{1 + \frac{1}{p_k}}^{1 + \frac{4\lambda_k^2}{\e ^2}} u^{-p_k/4}(u-1)^{-1/2}\d u \leq 
         \frac{(p_k+1)^{1/2}}{4|\lambda_k|}
         \int_{1 + \frac{1}{p_k}}^{1 + \frac{4\lambda_k^2}{\e ^2}} u^{-(p_k/4 + 1/2)}\d u.
        \end{equation}
        This is equal to:
        \begin{equation}\label{def: I_k}
        I_k(\e) :=
        \begin{cases}
        \frac{(p_k +1)^{1/2}}{4|\l_k|} \frac{1}{\frac{1}{2} - \frac{p_k}{4}} \sqbrace{\paren{1 + \frac{4\l_k^2}{\e^2}}^{\frac{1}{2} - \frac{p_k}{4}} - \paren{1 + \frac{1}{p_k}}^{\frac{1}{2} - \frac{p_k}{4}}} & \text{if }p_k \neq 2, \\ 
        \frac{3^{1/2}}{4|\l_k|}  \sqbrace{\log\paren{1 + \frac{4\l_k^2}{\e^2}} - \log\paren{1 + \frac{1}{2}}} & \text{if }p_k = 2. 
        \end{cases}
        \end{equation}
        For later analysis, we write all expressions in terms of $p_k$ and $\L_1$ as these are the natural units of the problem. This gives:
                \begin{equation}
        I_k(\e) =
        \begin{cases}
        \frac{(p_k^2 +p_k)^{1/2}}{2\L_1} \frac{1}{1 - \frac{p_k}{2}} \sqbrace{\paren{1 + \frac{4\L_1^2}{\e^2}\frac{1}{p_k}}^{\frac{1}{2}\paren{1 - \frac{p_k}{2}}} - \paren{1 + \frac{1}{p_k}}^{\frac{1}{2}\paren{1  - \frac{p_k}{2}}}} & \text{if }p_k \neq 2,\\ 
        \frac{3^{1/2} 2^{1/2}}{4\L_1}  \sqbrace{\log\paren{1 + \frac{4\L_1^2}{\e^2}} - \log\paren{1 + \frac{1}{2}}} & \text{if }p_k = 2. 
        \end{cases}
        \end{equation}
    \newline 
    \newline 
    Our final bound is thus:
    \begin{equation}
        \int_{\frac{1}{2\L_1}}^\frac{1}{\e} \prod_{k=1}^\infty\frac{1}{(1 + 4 \l_k^2 t^2)^{1/4}}\d t 
\leq \prod_{k=1}^{\infty}I_k(\e)^{1/p_k}.
    \end{equation}
    To analyze this product, we consider cases (in order of simplicity of analysis), depending on the size of $p_1$. We consider the following 3 cases, and show that they provide the respective rates in $\e$:
    \begin{enumerate}
        \item $p_1 > 2$ (Linear in $\e$)
        \item $1 \leq p_1 \leq 2$ (Sublinear in $\e$ or Linear in $\e$ if $q_1 > 0$)
        \item $p_1 = 2$ (Log-Linear in $\e$ or Linear in $\e$ if $q_1 > 0$)
    \end{enumerate}
    \begin{description}
        \item[Case 1: ($p_1 > 2$)] 
        In this case, $\frac{\l_1 ^2}{\L_1^2} = \frac{1}{p_1} < \frac{1}{2}$, e.g. the first term in the sum does not "dominate" the others. In this case, we obtain the best possible rate in $\e$ (linear), provided $p_1$ is bounded away from $2$. 
        \newline
        \newline 
        By monotonicity, we know that $p_k \geq 2$ for all $k$ as well.
        In such a case, $1 - p_k/2 < 1$ and so $I_k$ may be written more suggestively:
        $$
        \frac{(p_k^2 +p_k)^{1/2}}{2\L_1} \frac{1}{ \frac{p_k}{2} -1} \sqbrace{ \paren{1 + \frac{1}{p_k}}^{\frac{1}{2}\paren{1  - \frac{p_k}{2}}} - \paren{1 + \frac{4\L_1^2}{\e^2}\frac{1}{p_k}}^{\frac{1}{2}\paren{1 - \frac{p_k}{2}}}}.
        $$
        Using the fact that $1 -\frac{p_k}{2} < 0$ and $\frac{4\L_1^2}{\e^2} > 1$, we may estimate the second term in parentheses above as:
        \begin{equation} \label{eq: Lambda pk inequality}
        \paren{1 + \frac{4\L_1^2}{\e^2}\frac{1}{p_k}}^{\frac{1}{2}\paren{1 - \frac{p_k}{2}}} \geq \paren{\frac{4\L_1^2}{\e^2} + \frac{4\L_1^2}{\e^2}\frac{1}{p_k}}^{\frac{1}{2}\paren{1 - \frac{p_k}{2}}} = \paren{1 + \frac{1}{p_k}}^{\frac{1}{2}\paren{1- \frac{p_k}{2}}} \paren{\frac{4\L_1^2}{\e^2}}^{\frac{1}{2}\paren{1 - \frac{p_k}{2}}}
        \end{equation}
        Applying this inequality and factoring, we obtain the estimate:
        \begin{align*}
        \sqbrace{ \paren{1 + \frac{1}{p_k}}^{\frac{1}{2}\paren{1  - \frac{p_k}{2}}} - \paren{1 + \frac{4\L_1^2}{\e^2}\frac{1}{p_k}}^{\frac{1}{2}\paren{1 - \frac{p_k}{2}}}} &\leq \paren{1 + \frac{1}{p_k}}^{\frac{1}{2}\paren{1  - \frac{p_k}{2}}}\sqbrace{1 - \paren{\frac{4\L_1^2}{\e^2}}^{\frac{1}{2}\paren{1 - \frac{p_k}{2}}}} \\&\leq \sqbrace{1 - \paren{\frac{2\L_1}{\e}}^{\paren{1 - \frac{p_k}{2}}}}.
        \end{align*}
        On the other hand:
        \begin{align*}
       \frac{(p_k^2 +p_k)^{1/2}}{2\L_1} \frac{1}{ \frac{p_k}{2}- 1 } &= \frac{\paren{1 + \frac{1}{p_k}}^{1/2}}{\L_1}\frac{p_k}{p_k- 2} \\&\leq \paren{\frac{3}{2}}^{3/2} \frac{1}{\L_1}\frac{p_k}{p_k - 2} \\&= \paren{\frac{3}{2}}^{3/2} \frac{1}{\L_1}\frac{1}{1 - \frac{2}{p_k}} 
        \end{align*}
        We obtained the term $3/2$ simply by substituting $p_k = 2$ into the decreasing function $1 + \frac{1}{p_k}$. 
        Thus far, we have shown:
        \begin{equation}\label{eq: Indiv Ik Estimate}
        I_k(\e) \leq \paren{\frac{3}{2}}^{3/2} \frac{1}{\L_1}\frac{1}{1 - \frac{2}{p_k}}     \sqbrace{1 - \paren{\frac{2\L_1}{\e}}^{\paren{1 - \frac{p_k}{2}}}}.
        \end{equation}
        Taking products and powers to the $1/p_k$ gives an estimate:
        \begin{align*}
        \prod_{k=1}^{\infty} I_k(\e)^{1/p_k} &\leq \paren{\frac{3}{2}}^{3/2} \frac{1}{\L_1}\prod_{k=1}^{\infty}\paren{\frac{1}{1 - \frac{2}{p_k}}}^{1/p_k}     \sqbrace{1 - \paren{\frac{2\L_1}{\e}}^{\paren{1 - \frac{p_k}{2}}}}^{1/p_k} \\ 
        &= \paren{\frac{3}{2}}^{3/2} \frac{1}{\L_1}\prod_{k=1}^\infty J_k(\e)^{1/p_k},
        \end{align*}
        where we have set:
        \begin{equation}\label{eq: def Jk}
        J_k(\e) := \frac{1}{1-\frac{2}{p_k}}\sqbrace{1 - \paren{\frac{2\L_1}{\e}}^{\paren{1 - \frac{p_k}{2}}}}.
        \end{equation}
        Using that the term in square braces is always less than $1$, we have for $k \geq 3$:
        \begin{align}\label{Eq: k geq 3 inequality}
        \prod_{k=3}^{\infty}J_k(\e)^{1/p_k}  &= \paren{\frac{1}{1 - \frac{2}{p_k}}}^{1/p_k}\sqbrace{1 - \paren{\frac{2\L_1}{\e}}^{\paren{1 - \frac{p_k}{2}}}}^{1/p_k}  \\ 
        &\leq \prod_{k=3}^{\infty}\paren{\frac{1}{1 - \frac{2}{p_k}}}^{1/p_k} \\ 
        &\leq \prod_{k=3}^{+\infty} \paren{\frac{1}{1 - \frac{2}{p_3}}}^{1/p_k} \\ 
        &= \paren{\frac{1}{1 - \frac{2}{p_3}}}^{1 - \paren{\frac{1}{p_1} + \frac{1}{p_2}}}.
        \end{align}
       We obtained the penultimate inequality by plugging $p_3$ into the decreasing function $\frac{1}{1 - \frac{2}{x}}$ and then multiplying over $k$.
        As $\frac{1}{p_3} \leq \frac{1}{3}$, the above is bounded by:
        $$
        \paren{\frac{1}{1 - \frac{2}{3}}}^{1 - \paren{\frac{1}{p_1} + \frac{1}{p_2}}} = 3^{1 - \paren{\frac{1}{p_1} + \frac{1}{p_2}}} \leq 3.
        $$
        Thus:
        \begin{align*}
        \prod_{k=1}^{\infty}I_k(\e)^{1/p_k} &= \paren{\frac{3}{2}}^{3/2} \frac{1}{\L_1} \prod_{k=1}^\infty J_k(\e)^{1/p_k}  \\
        &\leq \paren{\frac{3}{2}}^{3/2} \sqbrace{J_1(\e)^{1/p_1}J_2(\e)^{1/p_2}}\prod_{k=3}^{+\infty}J_k(\e)^{1/p_k} \\&\leq 3 \paren{\frac{3}{2}}^{3/2} \frac{1}{\L_1} \sqbrace{J_1(\e)^{1/p_1}J_2(\e)^{1/p_2}}.
        \end{align*}
        Recalling the definition of $J_k(\e)$ from \eqref{eq: def Jk}
        gives that the entire product can be bounded by:
        \begin{align*}
          3\paren{\frac{3}{2}}^{3/2} \frac{1}{\L_1} \paren{\frac{1}{1 - \frac{2}{p_1}}}^{1/p_1} \paren{\frac{1}{1 - \frac{2}{p_2}}}^{1/p_2}    \sqbrace{1 - \paren{\frac{2\L_1}{\e}}^{\paren{1 - \frac{p_1}{2}}}}^{1/p_1}     \sqbrace{1 - \paren{\frac{2\L_1}{\e}}^{\paren{1 - \frac{p_2}{2}}}}^{1/p_2}.
        \end{align*}
        Recalling the definitions $A_{p,i}$ from the statement of Theorem \ref{Thm:Full Bound} gives the final estimate:
        \begin{align*}
        \frac{3^{5/2}}{2^{3/2}}\frac{1}{\L_1}\paren{\frac{1}{2}}^{\frac{1}{p_1} + \frac{1}{p_2}}A_{p,1}^{1/p_1} A_{p,2}^{1/p_2} \leq \frac{6}{\L_1} A_{p,1}^{1/p_1}A_{p,2}^{1/p_2},
        \end{align*}
        which is the desired estimate.
        \item[Case 2: ($1 \leq p_1 < 2$)]
        In such a case, the term corresponding to $\l_1$ "dominates" the remaining terms in the sum. Heuristically, we should expect slower rates in this case. As described in Section 4, we can obtain both sublinear and linear rates in $\e$ in this case. The linear rate is inversely proportional to $\paren{\L_1 \L_2}^{1/2}$, whereas the sublinear rate is inversely proportional to $\L_1$.
        \newline 
        \newline
        Finally, notice that in this case $p_1 < 2$ means that $p_k \geq 2$ for all $k$, as if this were not true we would have $1 \geq  \frac{1}{p_1} + \frac{1}{p_k} > \frac{1}{2} + \frac{1}{2} = 1$ for some $k$. 

         \begin{description}
            \item[Case 2.1: (Sublinear Concentration)]
            As $p_1 < 2$, we may write:
            \[
                    I_1(\e) = \frac{(p_1^2 +p_1)^{1/2}}{2\L_1} \frac{1}{1 - \frac{p_1}{2}} \sqbrace{\paren{1 + \frac{4\L_1^2}{\e^2}\frac{1}{p_1}}^{\frac{1}{2}\paren{1 - \frac{p_1}{2}}} - \paren{1 + \frac{1}{p_1}}^{\frac{1}{2}\paren{1  - \frac{p_1}{2}}}}.
                    \]
            For the term in square braces, we reason similarly (but essentially in the reverse fashion as $1 - \frac{p_1}{2} > 0$) as we did in the $p_1 \geq 2$ case (see equation \eqref{eq: Lambda pk inequality}) and obtain:
            \begin{align*}
            \sqbrace{\paren{1 + \frac{4\L_1^2}{\e^2}\frac{1}{p_1}}^{\frac{1}{2}\paren{1 - \frac{p_1}{2}}} - \paren{1 + \frac{1}{p_1}}^{\frac{1}{2}\paren{1  - \frac{p_1}{2}}}} &\leq\paren{1 + \frac{1}{p_1}}^{\frac{1}{2}\paren{1 - \frac{p_1}{2}}} \sqbrace{\paren{\frac{4\L_1^2}{\e^2}}^{\frac{1}{2}\paren{1 - \frac{p_1}{2}}} - 1} \\ 
            &\leq 2^{1/2}\sqbrace{\paren{\frac{2\L_1}{\e}}^{1 - \frac{p_1}{2}} -1 }.
            \end{align*}
            And as before:
            $$
            \frac{(p_1^2 +p_1)^{1/2}}{2\L_1} \frac{1}{1 - \frac{p_1}{2}} = \frac{\paren{1 + \frac{1}{p_1}}^{1/2}}{\L_1}\frac{p_1}{2 - p_1} \leq \frac{2^{1/2}}{\L_1}\frac{1}{\frac{2}{p_1} -1}.
            $$
            So in total:
            $$
            I_1(\e) \leq \frac{2}{\L_1}\frac{1}{\frac{2}{p_1} -1}\sqbrace{\paren{\frac{2\L_1}{\e}}^{1 - \frac{p_1}{2}} -1 } \leq \frac{4}{\L_1}\frac{1}{\frac{2}{p_1} -1}\sqbrace{\paren{\frac{2\L_1}{\e}}^{1 - \frac{p_1}{2}} -1 } .
            $$
            On the other hand, for the remaining $I_k(\e)$, using the fact that $p_k \geq 2$ and mimicking the approach of the previous case gives:
            \begin{align*}
            I_{k}(\e) &\leq \paren{\frac{3}{2}}^{3/2} \frac{1}{\L_1}\frac{1}{1 - \frac{2}{p_k}}     \sqbrace{1 - \paren{\frac{2\L_1}{\e}}^{\paren{1 - \frac{p_k}{2}}}} \\ 
            &:= \paren{\frac{3}{2}}^{3/2} \frac{1}{\L_1}J_k(\e) \\ 
            &\leq \frac{6}{\L_1} J_{k}(\e).
            \end{align*}
            By the same argument as in the first case as carried out in equation \eqref{Eq: k geq 3 inequality} we obtain:
            \[\prod_{k=2}^{\infty}I_k(\e)^{1/p_k} \leq \paren{\frac{6}{\L_1}}^{1 - \frac{1}{p_1}} J_2(\e)^{1/p_2}.
            \]
            Putting this together and again recalling the definition  of $J$ from \eqref{eq: def Jk} gives:
            \begin{align*}
                \prod_{k=1}^{\infty}I_k(\e)^{1/p_k} &= I_1(\e)^{1/p_1}\prod_{k=2}^{\infty}I_k(\e)^{1/p_k} \\
                &\leq \paren{\frac{6}{\L_1}}^{1 - \frac{1}{p_1}}\paren{\frac{6}{\L_1}}^{\frac{1}{p_1}}\sqbrace{J_1(\e)^{1/p_1}J_2(\e)^{1/p_2}} \\
                &\leq \frac{6}{\L_1} \paren{\frac{1}{\frac{2}{p_1} -1}}^{1/p_1}         \paren{\frac{1}{1 - \frac{2}{p_2}}}^{1/p_2}  \sqbrace{\paren{\frac{2\L_1}{\e}}^{1 - \frac{p_1}{2}} -1 }
   \sqbrace{1 - \paren{\frac{2\L_1}{\e}}^{\paren{1 - \frac{p_2}{2}}}}^{1/p_2} \\
   &= \frac{6}{\L_1}B_{p,1}^{1/p_1}A_{p,2}^{1/p_2}.
            \end{align*}
            which is the desired result.
            \item[Case 2.2: (Linear Concentration,  $q_2> 1$ i.e. $\l_3 \neq 0$)]
            In such a case, the $\l_1$ term still dominates the others, however there are enough nonzero terms in the sum to obtain a linear rate in $\e$. Thus, we need to derive a bound that simultaneously exploits the largeness of $\l_1$ while also including the contribution of smaller $\l_k$.  
            \newline 
            \newline 
            To estimate the integral, we have $\frac{1}{2\L_1} < \frac{1}{2\L_2} \leq \min \cbrace{\frac{1}{\e}, \frac{1}{2\L_2}}$, so similarly to the proof of Theorem \ref{Thm:Full Bound} (in particular, Step 2 in Section \ref{Sec: Proof}), we divide the region of integration and obtain for any $f$ nonnegative and measurable:
            $$
            \int_{\frac{1}{2\L_1}}^{\frac{1}{\e}} f(t) \d t \leq \int_{\frac{1}{2\L_1}}^{\frac{1}{2\L_2}} f(t)\d t + \mathbbm{1}{\{\e \leq 2 \L_2\}}\int_{\frac{1}{2\L_2}}^{\frac{1}{\e}} f(t)\d t .
            $$
            We define:
            $$
            \mathrm{I} :=  \int_{\frac{1}{2\L_1}}^{\frac{1}{\e}} \prod_{k=1}^{\infty} \frac{1}{(1 + 4 \l_k^2t^2)^{1/4}} \d t, \quad \mathrm{II}:= \int_{\frac{1}{2\L_2}}^{\frac{1}{\e}} \prod_{k=1}^{\infty} \frac{1}{(1 + 4 \l_k^2t^2)^{1/4}} \d t.
            $$
            For $\mathrm{I}$, we estimate:
            \begin{align*}
            \int_{\frac{1}{2 \L_1}}^{\frac{1}{2\L_2}} \prod_{k=1}^\infty \frac{1}{(1 + 4\l_k^2 t^2)^{1/4}} \d t 
            &\leq \int_{\frac{1}{2\L_1}}^{\frac{1}{2\L_2}} \frac{1}{(1 + 4\l_1^2t^2)^{1/4}} \d t \\
            &=\frac{1}{4|\l_1|}\int_{1 + \frac{1}{p_1}}^{1 + \frac{\l_1^2}{\L_2^2}} \frac{1}{u^{1/4}(u-1)^{1/2}} \d u \\
            &\leq \frac{(p_1^2 + p_1)^{1/2}}{4 \L_1}\int_{1 + \frac{1}{p_1}}^{1 + \frac{\l_1^2}{\L_2^2}} u^{-3/4} \d u \\
            &= \frac{(p_1^2 +p_1)^{1/2}}{\L_1}\sqbrace{\paren{1 + \frac{\l_1^2}{\L_2^2}}^{1/4} - \paren{1 + \frac{1}{p_1}}^{1/4} } 
            \\
            &\leq \frac{(p_1^2 +p_1)^{1/2}}{\L_1}\paren{1 + \frac{1}{p_1}}^{1/4}\sqbrace{\frac{\L_1^{1/2}}{\L_2^{1/2}} - 1} 
            \\& \leq \frac{6^{1/2}2^{1/4}}{\L_1}\sqbrace{\frac{\L_1^{1/2}}{\L_2^{1/2}} - 1} \\
            &\leq \frac{3}{\L_1}\sqbrace{\frac{\L_1^{1/2}}{\L_2^{1/2}} - 1} \\ 
             &\leq \frac{3}{(\L_1\L_2)^{1/2}}.
            \end{align*} 
            For $\mathrm{II}$, we begin with the estimate:
            \begin{align}
            \prod_{k=1}^{\infty} \frac{1}{(1 + 4 \l_k^2t^2)^{1/4}} &\leq \frac{1}{4^{1/4}|\l_1|^{1/2}t^{1/2}}\prod_{k=2}^{\infty} \frac{1}{(1 + 4 \l_k^2t^2)^{1/4}} \\
            &= \prod_{k=2}^\infty \paren{\frac{1}{(1 + 4 \l_k^2t^2)^{1/4}}\frac{1}{ (4^{1/4}|\l_1|^{1/2}t^{1/2})^{1/q_k}}}.
            \end{align}
            Applying H\"older's inequality to the final expression with weights $q_k$, $k \geq 2$ gives:
            \begin{align*}
            \mathrm{II} &\leq \prod_{k=2}^{\infty} \paren{\int_{\frac{1}{2\L_2}}^{\frac{1}{\e}} \frac{1}{4^{1/4}|\l_1|^{1/2} t^{1/2}(1 + 4\l_k^2t^2)^{q_k/4}} \d t}^{1/q_k} \\ &= \frac{1}{4^{1/4}|\l_1|^{1/2}}\prod_{k=2}^{\infty} \paren{\int_{\frac{1}{2\L_2}}^{\frac{1}{\e}} \frac{1}{t^{1/2}(1 + 4\l_k^2t^2)^{q_k/4}} \d t}^{1/q_k} .
            \end{align*}
            For the inner integrals, we proceed as we did in the $p_1 > 2$ case (see equations \eqref{eq: int substitution}, \eqref{eq: poly sub}) and thereafter), making the usual substitution $u = 1+ 4 \l_k^2 t^2$  and obtain:
            \begin{align*}
            \frac{1}{4|\l_k|}\int_{1 + \frac{1}{q_k}}^{1 + \frac{4 \l_k^2}{\e^2}}\frac{1}{\frac{(u-1)^{1/4}}{(2|\l_k|)^{1/2}} u^{q_k/4}} \frac{1}{(u-1)^{1/2}} \d u &= \frac{1}{2^{3/2}|\l_k|^{1/2}}\int_{1 + \frac{1}{q_k}}^{1 + \frac{4 \l_k^2}{\e^2}} u^{-q_k/4}(u-1)^{-3/4} \d u 
            \\
            &\leq \frac{(q_k + 1)^{3/4}}{2^{3/2}|\l_k|^{1/2}}\int_{1 + \frac{1}{q_k}}^{1 + \frac{4 \l_k^2}{\e^2}} u^{-q_k/4 -3/4} \d u . 
            \end{align*}
            Rewriting the constant in front in terms of $q_k$, $\L_2$ gives:
            $$
            \frac{(q_k +1)^{3/4}}{2^{3/2}\frac{\L_2^{1/2}}{q_k^{1/4}}} = \frac{(q_k^{4/3} + q_k^{1/3})^{3/4}}{2^{3/2}\L_2^{1/2}}.
            $$
            Because $q_k > 1$, we must have $q_k/ 4 + 3/4 > 1 $ and so the integrals above can be written as:
            $$
            \frac{1}{ \frac{q_k}{4} - \frac{1}{4}}\sqbrace{\paren{1 + \frac{1}{q_k}}^{\frac{1}{4} - \frac{q_k}{4}} - \paren{1+\frac{4 \l_k^2}{\e^2}}^{\frac{1}{4} - \frac{q_k}{4}}} = \frac{4}{q_k - 1}\sqbrace{\paren{1 + \frac{1}{q_k}}^{\frac{1}{4} - \frac{q_k}{4}} - \paren{1+\frac{4\L_2^2 \frac{1}{q_k}}{\e^2}}^{\frac{1}{4} - \frac{q_k}{4}}}.
            $$
            Setting
            \begin{align*}
            J_k(\e) &:= \frac{(q_k + 1)^{3/4}}{2^{3/2}|\l_k|^{1/2}}\int_{1 + \frac{1}{q_k}}^{1 + \frac{4 \l_k^2}{\e^2}} u^{-q_k/4 -3/4} \d u 
            \\
            &=   \frac{4(q_k^{4/3} + q_k^{1/3})^{3/4}}{2^{3/2}\L_2^{1/2}}\frac{1}{1 - q_k} \sqbrace{\paren{1 + \frac{1}{q_k}}^{\frac{1}{4} - \frac{q_k}{4}} - \paren{1+\frac{4\L_2^2 \frac{1}{q_k}}{\e^2}}^{\frac{1}{4} - \frac{q_k}{4}}}.
            \end{align*}
            we obtain:
            \[
            \mathrm{II} \leq \frac{1}{4^{1/4}|\l_1|} \prod_{k=2}^{\infty} J_k(\e)^{1/q_k} = \frac{p_1^{1/2}}{4^{1/4}\L_1} \prod_{k=2}^{\infty} J_k(\e)^{1/q_k}.
            \]
            To estimate each $J_k$ individually, we proceed similarly as we did from equation \eqref{eq: Lambda pk inequality} and onwards. For the term in square brackets:
            \[
            \sqbrace{\paren{1 + \frac{1}{q_k}}^{\frac{1}{4} - \frac{q_k}{4}} - \paren{1+\frac{4\L_2^2 }{\e^2}\frac{1}{q_k}}^{\frac{1}{4} - \frac{q_k}{4}}} \leq  \paren{1 + \frac{1}{q_k}}^{\frac{1}{4} - \frac{q_k}{4}}\sqbrace{1 - \paren{\frac{2\L_2}{\e}}^{\frac{1}{2} - \frac{q_k}{2}}} \leq  \sqbrace{1 - \paren{\frac{2\L_2}{\e}}^{\frac{1}{2} - \frac{q_k}{2}}}.
            \]
            And for the coefficient in front:
            $$
             \frac{4(q_k^{4/3} + q_k^{1/3})^{3/4}}{2^{3/2}\L_2^{1/2}} \frac{1}{q_k -1}  \leq 4\frac{2^{3/4}q_k}{2^{3/2}\L_2^{1/2}}\frac{1}{q_k  - 1} = \frac{4}{2^{3/4}} \frac{1}{\L_2^{1/2}} \frac{1}{1 - \frac{1}{q_k} }.
            $$
            By the same monotonicity argument made in the $p_1 > 2$ case (see equation \eqref{Eq: k geq 3 inequality}) :
            $$
            \frac{4}{2^{3/4}} \frac{1}{\L_2^{1/2}} \frac{1}{1 - \frac{1}{q_k}}\sqbrace{1 - \paren{\frac{2\L_2}{\e}}^{\frac{1}{2} - \frac{q_k}{2}}} \leq             \frac{4}{2^{3/4}} \frac{1}{\L_2^{1/2}} \frac{1}{1 - \frac{1}{q_2}}\sqbrace{1 - \paren{\frac{2\L_2}{\e}}^{\frac{1}{2} - \frac{q_2}{2}}}.
            $$
            So, it follows:
            \begin{align*}
            \mathrm{II} &\leq \frac{p_1^{1/2}}{4^{1/4}\L_1^{1/2}} \prod_{k=2}^\infty I_k(\e)^{1/q_k} \\ 
            &\leq \frac{4^{3/4}} {2^{3/4}} \frac{p_1^{1/2}}{\L_1^{1/2}} \frac{1}{\L_2^{1/2}} \frac{1}{1 - \frac{1}{q_2}}\sqbrace{1 - \paren{\frac{2\L_2}{\e}}^{\frac{1}{2} - \frac{q_2}{2}}} \\
            &\leq \frac{2^{5/4}}{(\L_1\L_2)^{1/2}}  \frac{1}{1 - \frac{1}{q_2}}\sqbrace{1 - \paren{\frac{2\L_2}{\e}}^{\frac{1}{2} - \frac{q_2}{2}}} \\
            &\leq \frac{3}{(\L_1\L_2)^{1/2}}  \frac{1}{1 - \frac{1}{q_2}}\sqbrace{1 - \paren{\frac{2\L_2}{\e}}^{\frac{1}{2} - \frac{q_2}{2}}}.
            \end{align*}
            Adding the bounds for $\mathrm{I}, \mathrm{II}$  together and replacing universal constants by their maxima/nicer numbers gives:
            $$
            \frac{3}{(\L_1\L_2)^{1/2}}\paren{1   +  \mathbbm{1}{\{\e \leq 2 \L_2\}}\frac{1}{1 - \frac{1}{q_2}}\sqbrace{1 - \paren{\frac{2\L_2}{\e}}^{\frac{1}{2} - \frac{q_2}{2}}}}.
            $$
            as we desire.
            \item[Case 3:(Linear-times-Logarithmic Concentration, $p_1 = 2$)]  This is the threshold case in which we obtain a linear-times-log rate. We write: 
            \begin{align*}I_1(\e)^{1/p_1} &=  \paren{\frac{3^{1/2}2^{1/2}}{4\L_1}  \sqbrace{\log\paren{1 + \frac{4\L_1^2}{\e^2}\frac{1}{2}} - \log\paren{1 + \frac{1}{2}}}}^{1/2} \\ &\leq \paren{\frac{3^{1/2}2^{1/2}}{4\L_1}  \sqbrace{\log\paren{\frac{4\L_1^2}{\e^2}\sqbrace {1 + \frac{1}{2}}} - \log\paren{1 + \frac{1}{2}}}}^{1/2} \\
            &= \paren{\frac{3^{1/2}2^{1/2}}{4\L_1}  \log\paren{\frac{4\L_1^2}{\e^2}}}^{1/2} \\ 
            &= \paren{\frac{3^{1/2}2^{1/2}}{2\L_1}  \log\paren{\frac{2\L_1}{\e}}}^{1/2}.
            \end{align*}
            Next, we know that for the remaining $p_k$ we must have $p_k \geq 2$. We briefly break into two short subcases, the first being nearly trivial.  
            \begin{description}
                \item[Case 3.1: $(p_2=2)$] In this case, we have $p_2 = 2$ as well, which tells us that $\l_1 = \l_2$ and $\l_3$ onwards are equal to $0$. In this case:
                $$
                \prod_{k=1}^\infty I_k(\e)^{1/p_k} = I_1(\e)^{1/2}I_2(\e)^{1/2} = I_1(\e) \leq \frac{3^{1/2}2^{1/2}}{2\L_1}  \log\paren{\frac{2\L_1}{\e}} \leq \frac{3}{\L_1} \log\paren{\frac{2\L_1}{\e}}.
                $$
            \item[Case 3.2: $(p_2 > 2)$] In this case, (porting over calculations made in the $p_1 > 2$ case) we have for any $k$:
            $$
              \prod_{k=2}^{\infty}I_k(\e) \leq \paren{\frac{4}{\L_1}}^{1 - \frac{1}{p_1}} \paren{\frac{1}{1 - \frac{2}{p_2}}}^{1/p_2}     \sqbrace{1 - \paren{\frac{2\L_1}{\e}}^{\paren{1 - \frac{p_2}{2}}}}^{1/p_2}.
            $$
            We write:
            $$
            \frac{1}{1 - \frac{2}{p_2}}     \sqbrace{1 - \paren{\frac{2\L_1}{\e}}^{\paren{1 - \frac{p_2}{2}}}} = \frac{\frac{p_2}{2}}{\frac{p_2}{2} - 1} \sqbrace{1 - \paren{\frac{\e}{2\L_1}}^{\frac{p_2}{2} - 1}} = 
 - \frac{h\paren{\frac{p_2}{2} - 1} - h(0)}{ \frac{p_2}{2} -1}
            $$
            where $h(x) = \paren{\frac{\e}{2\L_1}}^{x} = \exp\paren{x \log\paren{\frac{\e}{2\L_1}}}$. By the mean-value theorem:
            $$
            -\frac{h\paren{\frac{p_2}{2} - 1} - h(0)}{ \frac{p_2}{2} -1} \leq \log\paren{\frac{2\L_1}{\e}}.
            $$
            Hence:
            $$
            \paren{\frac{4}{\L_1}}^{1 - \frac{1}{p_1}} \paren{\frac{1}{1 - \frac{2}{p_2}}}^{1/p_2}     \sqbrace{1 - \paren{\frac{2\L_1}{\e}}^{\paren{1 - \frac{p_2}{2}}}}^{1/p_2} \leq \paren{\frac{4}{\L_1}}^{1 - \frac{1}{p_1}} \paren{p_2\log\paren{\frac{2\L_1}{\e}}}^{1/p_2}.
            $$
            Multiplying by $I_1^{1/p_1}$ and taking the maximum of all constants gives a bound:
            $$
            \frac{4}{\L_1} p_2^{1/p_2} \log\paren{\frac{2\L_1}{\e}}^{\frac{1}{2} + \frac{1}{p_2}} \leq \frac{4 e^{1/e}}{\L_1}\log\paren{\frac{2\L_1}{\e}}^{\frac{1}{2} + \frac{1}{p_2}} \leq \frac{6}{\L_1}\log\paren{\frac{2\L_1}{\e}}^{\frac{1}{2} + \frac{1}{p_2}}
            $$
            as desired.
            \end{description}
        \end{description}
    \end{description}
    \end{proof}
    \section{Proof that Theorem \ref{Thm:Full Bound} implies Theorem \ref{Thm:User Friendly Bound}}
    \label{Pf: Thm 1 Implies Thm 2}

    \begin{proof}
    We break into cases, again depending on $p_1$.
    \begin{description}
        \item[Case 1 $(p_1 \geq 2)$:] In this case, we have:
        $$
          \sqbrace{1 - \paren{\frac{2\L_1}{\e}}^{\paren{1 - \frac{p_1}{2}}}}^{1/p_1}     \sqbrace{1 - \paren{\frac{2\L_1}{\e}}^{\paren{1 - \frac{p_2}{2}}}}^{1/p_2} \leq 1
        $$
        as $\frac{2\L_1}{\e} > 1$ and $1 - \frac{p_1}{2} < 0$.
        \newline 
        \newline 
        Setting $p_1^*$ to satisfy $\frac{1}{p_1} + \frac{1}{p_1^*} = 1$, we have that $\frac{1}{p_2} \leq \frac{1}{p_1^*}$ and so for any $x \geq 1$ we must have $x^{1/p_2} \leq x^{1/p*}$. Thus:
        \begin{align*}
        &\frac{1}{\L_1} \paren{\frac{1}{1 - \frac{2}{p_1}}}^{1/p_1} \paren{\frac{1}{1 - \frac{2}{p_2}}}^{1/p_2}  \sqbrace{1 - \paren{\frac{2\L_1}{\e}}^{\paren{1 - \frac{p_1}{2}}}}^{1/p_1}     \sqbrace{1 - \paren{\frac{2\L_1}{\e}}^{\paren{1 - \frac{p_2}{2}}}}^{1/p_2}  \\
        &\leq \frac{1}{\L_1} \paren{\frac{1}{1 - \frac{2}{p_1}}}^{1/p_1} \paren{\frac{1}{1 - \frac{2}{p_2}}}^{1/p_2} \\ 
        &\leq \frac{1}{\L_1} \paren{\frac{1}{1 - \frac{2}{p_1}}}^{1/p_1} \paren{\frac{1}{1 - \frac{2}{p_2}}}^{1/p*} \\
        &\leq \frac{1}{\L_1} \max\cbrace{\frac{1}{1-\frac{2}{p_1}},\frac{1}{1-\frac{2}{p_2}}} = \frac{1}{\L_1} \frac{1}{1 - \frac{2}{p_1}}
        \end{align*}
        as desired.
        \item [Case 2: $(p_1 = 2)$]
        $\mathcal{J}, \mathcal{I}$ agree exactly in this case so there is nothing to prove. 
        \item[Case 3: $(1 \leq p_1 < 2)$] This case is analogous to $p_1 \geq 2$ case. We have that:
        $$\sqbrace{\paren{\frac{2\L_1}{\e}}^{1 - \frac{p_1}{2}} -1 }^{1/p_1}
   \sqbrace{1 - \paren{\frac{2\L_1}{\e}}^{1 - \frac{p_2}{2}}}^{1/p_2} \leq \paren{\frac{2\L_1}{\e}}^{\frac{1}{p_1} - \frac{1}{2}}
        $$
        And:
        $$
        \frac{1}{\L_1}\paren{\frac{1}{\frac{2}{p_1} -1}}^{1/p_1}         \paren{\frac{1}{1 - \frac{2}{p_2}}}^{1/p_2} \leq \frac{1}{\L_1}\paren{\frac{1}{\frac{2}{p_1} -1}}^{1/p_1}         \paren{\frac{1}{1 - \frac{2}{p^*}}}^{1/p*}
        $$
        where $p^*$ is defined as before. The reason for this is that:
        $$
        x \mapsto \paren{\frac{1}{1 - \frac{2}{x}}}^{1/x}
        $$
        is decreasing on the interval $[2, +\infty)$. Simplifying gives:
        $$
        \paren{\frac{1}{1 - \frac{2}{p^*}}}^{1/p*} = \paren{\frac{1}{\frac{1}{p_1} - \frac{1}{p^*}}}^{1/p^*} = \paren{\frac{1}{\frac{2}{p_1} - 1}}^{1/p^*} 
        $$
        So that in total, we obtain the estimate
        $$ \paren{\frac{1}{\frac{2}{p_1} - 1}}^{1/p_1} \paren{\frac{1}{\frac{2}{p_1} - 1}}^{1/p^*} \paren{\frac{2\L_1}{\e}}^{\frac{1}{p_1} - \frac{1}{2}} = \paren{\frac{1}{\frac{2}{p_1} - 1}}\paren{\frac{2\L_1}{\e}}^{\frac{1}{p_1} - \frac{1}{2}}$$
        which is what we desired.
        \item[Case 4 $(1 \leq p_1 < 2, q_1 >) $:]
        We have:
        \begin{align*}
        \frac{1}{(\L_1\L_2)^{1/2}}\paren{1   +  \mathbbm{1}{\{\e \leq 2 \L_2\}}\frac{1}{1 - \frac{1}{q_2}}\sqbrace{1 - \paren{\frac{2\L_2}{\e}}^{\frac{1}{2} - \frac{q_2}{2}}}}  &\leq \frac{1}{(\L_1\L_2)^{1/2}}\paren{1   +  \mathbbm{1}{\{\e \leq 2 \L_2\}}\frac{1}{1 - \frac{1}{q_2}}} \\&\leq \frac{1}{(\L_1\L_2)^{1/2}}\paren{1   + \frac{1}{1 - \frac{1}{q_2}}}
        \end{align*}
        as we need.
    \end{description}
    \end{proof}
    \section{Phase Transition for the Square of the Mean Statistic}
    Let $X, X_1, \cdots, X_n \overset{IID}{\sim} P_n$ with $\E[X] = \theta_n$ and $\E[(X-\theta_n)] = \sigma^2_n$. We consider the $U$-statistic with kernel $h = x-\theta_n$, e.g.:
    \begin{equation}
    U_n := \frac{1}{{n \choose 2}}\sum_{1 \leq i < j \leq n} h(X_i, X_j) = \frac{1}{{n \choose 2}}\sum_{1 \leq i < j \leq n}(X_i-\theta_n)(X_j - \theta_n)
    \end{equation}
    To see where and why a phase-transition in the limiting behavior of this statistic occurs, we decompose $U_n$ as follows - we first write:
    $$
    X_iX_j - \theta_{n}^2 = (X_i - \theta_n)(X_j - \theta_n) + \theta_n(X_j - \theta_n) + (X_i - \theta_n)\theta_n +\theta_n^2  - \theta_n^2 
    $$
    Summing gives:
    $$
    \sum_{\uind{i}{j}{n}} (X_i - \theta_n)(X_j - \theta_n) + \theta_n(X_j - \theta_n) + (X_i - \theta_n)\theta_n 
    $$
    For the second term:
    $$
    \sum_{\uind{i}{j}{n}} \theta_n(X_j - \theta_n) = \sum_{i=1}^{n} \sum_{ j = i + 1}^{n} \theta_n(X_j - \theta_n) = \sum_{j=2}^{n}\sum_{i = 1}^{j-1}\theta_n(X_j - \theta_n) = \theta_n \sum_{j=2}^{n} (j-1)(X_j - \theta_n)
    $$
    Even more simply, for the third:
    $$
    \sum_{\uind{i}{j}{n}} (X_i - \theta_n)\theta_n = \sum_{ i =1}^{n-1} \sum_{ j = i+1}^{n} (X_i - \theta_n)\theta_n = \theta_n \sum_{i = 1}^{n - 1}(n- (i + 1))(X_i - \theta_n) = \theta_n \sum_{i=2}^{n}(n-i)(X_i - \theta_n)
    $$
    So:
    $$
     \sum_{\uind{i}{j}{n}}X_iX_j = \sum_{\uind{i}{j}{n}}(X_i - \theta_n)(X_j - \theta_n) + \theta_n(n-1) \sum_{k=2}^{n}(X_k - \theta_n)
    $$
    and finally arrive at:
    \begin{equation}
            U_n = \underbrace{\frac{1}{{n \choose 2}}\sum_{\uind{i}{j}{n}}(X_i - \theta_n)(X_j - \theta_n)}_{:=U^C_n} +  \underbrace{\frac{2\theta_n}{n}\sum_{k=2}^{n}(X_k - \theta_n)}_{:=U_n^G}
    \end{equation}
        The second term, being the sum of IID random variables, can be thought of as the Gaussian part of the statistic, while the first term can be thought of as the quadratic/chi-square part.
\subsubsection{Analysis of Limiting Behavior}
\label{Limiting Behavior of Un}
    Let us examine this bound in some typical scenarios and use some heuristics to understand which regimes lead to Gaussian behavior and which regimes lead to chi-square behavior as $n \to \infty$. Throughout, we make the following two assumptions on the distribution $P_n$:
        \begin{description}
    \item[Assumption 1 (CLT):]      \begin{equation}
        \frac{1}{n^{1/2}}\sum_{k=1}^{n} \frac{X_k - \theta_n}{\sigma_n} \underset{D}{\to} \Normal(0,1)
        \end{equation}
     \item[Assumption 2 (LLN):]  
     \begin{equation}
        \frac{1}{n}\sum_{k=1}^{n} \frac{(X_k - \theta_n)^2}{\sigma_n} \underset{\bb{P}}{\to} 1
        \end{equation}
    \end{description}
    Note that these assumptions both hold true in the classical case when $\theta_n = \theta$, $\sigma_n = \sigma > 0$, e.g. both quantities are fixed.
    \newline 
    \newline 
    With these assumptions in hand, we can show that the statistics $n^{1/2}U_n$, $nU_n$ have tractable limits in a variety of cases, much like what occurs in the classical CLT for U-statistics. 
    \begin{description}
        \item[\textbf{Case 1:} $\theta_n \sigma_n \to C_1 \neq 0$   \textbf{and} $\sigma_n = o(n^{1/4})$]
        In this case, we will show that $n^{1/2}U_n \underset{D}{\to} 2C_1Z$. 
        \newline 
        \newline 
        We first show that $n^{1/2}U_n^C$ converges to $0$ in probability. We have:
        \begin{align}
            n^{1/2}U_n^C = \frac{2}{n^{1/2}(n-1)}\sqbrace{\paren{\sum_{i=1}^{n}X_i - \theta_n}^2 - \sum_{i=1}^{n}(X_i - \theta_n)^2} 
        \end{align}
        For the first term, we write:
        $$
        \frac{2}{n^{1/2}(n-1)}\paren{\sum_{i=1}^{n}X_i - \theta_n}^2 = \frac{2\sigma_n^2}{n^{1/2}} \paren{\frac{1}{(n-1)^{1/2}} \sum_{i=1}^{n}\frac{X_i - \theta_n}{\sigma_n}}^2 \underset{\bb P}{\to} 0  
        $$ 
        by the growth rate assumption on $\sigma_n$ and the fact that the term in parentheses converges in distribution to a Gaussian. For the second term:
        $$
        \E\sqbrace{\frac{2}{n^{1/2}(n-1)}\sum_{i=1}^{n}(X_i - \theta_n)^2} = \frac{n\sigma_n^2}{n^{1/2}(n-1)} \to 0 
        $$
        by the assumption on $\sigma_n$. Thus, $n^{1/2}U_n^C$ converges to $0$ in probability. For the second term, we have:
    $$
    n^{1/2} U_n^G = 2\theta_n \sigma_n \cdot \frac{1}{n^{1/2}} \sum_{k=2}^{n} \frac{(X_k - \theta_n)}{\sigma_n} \underset{D}{\to} 2C_1Z
    $$
    as desired.
        \item[Case 2:  $\theta_n\sigma_n \to C_1 \neq 0$ and $\frac{\sigma_n^2}{n^{1/2}} \to C_2 > 0$]
        In this case, we will show that $n^{1/2}U_n$ has a limit of the form:
        $$
        2C_1 Z_1 + 2C_2(Z_1^2 - 1)
        $$
        e.g. a non-central chi-square random variable.
        Repeating the same calculations as above gives us that:
        $$
        n^{1/2}U_n^C = \frac{2}{n^{1/2}(n-1)}\sqbrace{\paren{\sum_{i=1}^{n}X_i - \theta_n}^2 - \sum_{i=1}^{n}(X_i - \theta_n)^2} 
        $$
        Again for the first term:
        $$
        \frac{2}{n^{1/2}(n-1)}\paren{\sum_{i=1}^{n}X_i - \theta_n}^2 = \frac{2\sigma_n^2}{n^{1/2}} \paren{\frac{1}{(n-1)^{1/2}} \sum_{i=1}^{n}\frac{X_i - \theta_n}{\sigma_n}}^2 \underset{D}{\to} 2C_2Z^2
        $$
        By assumption and the continuous mapping theorem. For the second term:
        $$
        \frac{2}{n^{1/2}(n-1)}\sum_{i=1}^{n}(X_i - \theta_n)^2 = \frac{2\sigma_n^2}{n^{1/2}} \frac{1}{(n-1)}\sum_{k=1}^{n} \frac{(X_i - \theta_n)}{\sigma_n^2} \underset{\bb{P}}{\to} 2C_2
        $$
        Finally, for $n^{1/2}U_n^G$:
        $$
n^{1/2} U_n^G = 2\theta_n \sigma_n \cdot \frac{1}{n^{1/2}} \sum_{k=2}^{n} \frac{(X_k - \theta_n)}{\sigma_n} \underset{D}{\to} 2C_1 Z          
$$
Adding these two limits gives us our desired result. 
\item[Case 3:  $\theta_n\sigma_n \to 0 $ and $\frac{\sigma_n^2}{n^{1/2}} \to C_2 > 0$]
        In this case, we will show that $n^{1/2}U_n$ has a limit of the form:
        $$
        2C_2(Z_1^2 - 1)
        $$
        Repeating the same calculations as above gives us that:
        $$
        n^{1/2}U_n^C = \frac{2}{n^{1/2}(n-1)}\sqbrace{\paren{\sum_{i=1}^{n}X_i - \theta_n}^2 - \sum_{i=1}^{n}(X_i - \theta_n)^2} 
        $$
        Again for the first term:
        $$
        \frac{2}{n^{1/2}(n-1)}\paren{\sum_{i=1}^{n}X_i - \theta_n}^2 = \frac{2\sigma_n^2}{n^{1/2}} \paren{\frac{1}{(n-1)^{1/2}} \sum_{i=1}^{n}\frac{X_i - \theta_n}{\sigma_n}}^2 \underset{D}{\to} 2C_2Z^2
        $$
        By assumption and the continuous mapping theorem. For the second term:
        $$
        \frac{2}{n^{1/2}(n-1)}\sum_{i=1}^{n}(X_i - \theta_n)^2 = \frac{2\sigma_n^2}{n^{1/2}} \frac{1}{(n-1)}\sum_{k=1}^{n} \frac{(X_i - \theta_n)}{\sigma_n^2} \underset{\bb{P}}{\to} 2C_2
        $$
        Finally, for $n^{1/2}U_n^G$:
        $$
n^{1/2} U_n^G = 2\theta_n \sigma_n \cdot \frac{1}{n^{1/2}} \sum_{k=2}^{n} \frac{(X_k - \theta_n)}{\sigma_n} \underset{D}{\to} 0          
$$
Thus, we obtain a distributional limit of the form:
$$
2C_2(Z^2 - 1)
$$
which is a central chi-square random variable.
    \end{description}
    \section{Eigenfunction and Eigenvalue Computations}
    In this section, we provide diagonalization computations for the two examples of U-Statistics we mention and discuss in our applications section \ref{Sec:Applications}. 
    \label{Sec: Eigenvalue Calculations}
    \subsection{Square of the Mean and General Rank-1 U-Statistic}
        Following the notation of section \ref{sec: rank 1 application} we write $$\E[X] := \theta_n := \theta$$ and 
        $$ \mathrm{Var}(X) := \sigma_n := \sigma$$ i.e. we suppress the dependence on $n$ to simplify notation throughout.
        \newline 
        \newline 
        Let us compute the approximating Gaussian quadratic form $W_n$ associated to $U_n$. For this, we need to compute the functions $H_1, H_2$, and then compute the eigenvalues and eigenfunctions of the Hilbert Schmidt operator corresponding to $H_2$. In this case:
        $$
        h(x,y) = xy - \theta^2
        $$
        So that:
    $$
    H_1(x) = \E[h(X_1,X_2) | X_1 = x] = \E[xX_2 - \theta^2] = x \theta - \theta^2
    $$
    and:
    $$
    H_2(x,y) = h(x,y) - H_1(x) - H_1(y) = xy - \theta(x+y) - \theta^2
    $$
    Next, we compute the eigenfunctions of $H_2$ with nonzero eigenvalue. Let $\psi \in L^{2}(\bb R, P_n)$ be arbitrary,  then the eigenfunction equation is:
    $$
    c\psi(x) = \E[H_2(x, X)\psi(X)]
    $$
    We compute:
    \begin{align*}
    \E[H_2(x,X)\psi(X)] &= \E[(xX - \theta(x + X) - \theta^2)\psi(X)] \\
    &= x\E[X\psi(X)] - \theta x \E[\psi(X)] -\theta\E[X\psi(X)] - \theta^2\E[\psi(X)] \\
    &= x\E[X\psi(X)]  -\theta\E[X\psi(X)] \\
    &= \E[X\psi(X)](x-\theta)
    \end{align*}
    Where here we used the fact that $\E[\psi(X)] = 0$. The final expression is a linear function in $x$, thus the image of $H_2$ is a subset of the linear functions, and we may set $\psi(x) = ax + b$ for $a,b \in \mathbb{R}$. Using orthonormality, we obtain:
    \begin{equation*}
            \E[aX + b] = 0 \implies a \theta + b = 0 \text{ and } \E[(aX + b)^2] =1 \implies a^2(\theta^2 + \sigma^2) + 2ab\theta + b^2 = 1
    \end{equation*}
     Using these equations gives:
    $$
    a = \frac{1}{\sigma}, b = -\frac{\theta}\sigma
    $$
    So that $$\phi_1(x) = \psi(x) = \frac{(x-\theta)}{\sigma}$$
    Thus, the image of $H_2$ is one-dimensional, and spanned by $\phi_1$. Having found the sole nontrivial eigenfunction, it remains to find the eigenvalue $c$. The eigenvalue equation is:
    \begin{align*}
c \frac{(x-\theta)}{\sigma} &= \E\sqbrace{H_2(x,X)\psi(X)} \\
&=\E[X\psi(X)](x-\theta) 
    \end{align*}
    So:
    $$
    \frac{c}{\sigma} = \E[X\psi(X)] = \E\sqbrace{\frac{X(X-\theta)}{\sigma}} \implies c = \E[X^2] - \theta^2 = \sigma^2
    $$
    This completes the process of finding eigenvalues and eigenfunctions. In summary, using the notation of Section \ref{sec:limi-dist-U-stat}, we obtain:
    \begin{equation}
        \phi_1(x) = \paren{\frac{x-\theta}{\sigma}}, \quad b_1 = \sigma
    \end{equation}
    The final step in the analysis is to represent $H_1(x), H_2(x,y)$ in terms of these functions and add another eigenfunction if needed. For $H_2$, we obtain the representation:
    Thus:
    $$
    H_2(x,y) = b_1\phi_1(x)\phi_1(y) =  \sigma^2\paren{\frac{x - \theta}{\sigma}}\paren{\frac{y - \theta}{\sigma}}
    $$
    For $H_1$, we seek a representation of the form:
    $$
    H_1(x) = x\theta = a_0\phi_0(x) + a_1 \phi_1(x) 
    $$
    where $\phi_0$ is a function to be chosen orthonormal to $\phi_1$. We observe:
    $$
    H_1(x) = x\theta-\theta^2 = \theta(x-\theta) = \theta \sigma \frac{x-\theta}{\sigma} = \theta\sigma \phi_1(x)
    $$
    Thus, again with notation of section \ref{sec:limi-dist-U-stat}, we obtain:
    \begin{equation}
        \phi_1(x) = \frac{x-\theta}{\sigma}, \quad a_1 = \theta \sigma
    \end{equation}
    This completes the analysis.
    \subsection{Student's T Statistic Kernel}
    We have:
    $$
    H_1(x) = \E[h(x,X)] = \E[xX^2 + x^2X] = x
    $$
    So:
    $$
    H_2(x,y) = h(x,y) - H_1(x) - H_1(y) = xy^2 + x^2 y - x - y= x(y^2-1) + y(x^2-1)
    $$
    For any function $\phi$, the integral operator gives:
    \begin{equation}
    \E[H_2(x,X)\phi(X)] = x\E[(X^2-1) \phi(X)] + (x^2 -1)\E[X\phi(X)] \label{Eq:OperatorEquation}
    \end{equation}
    Thus, the image of $H_2$ is contained in $\mathrm{span}\{x, x^2-1\}$ (which is a linearly independent set of functions), and it is at most a rank $2$ operator. In particular, any eigenfunction must lie in this set as well. Restricted to this set, the operator can be written in matrix form with respect to the basis as:
    $$
    A =\begin{pmatrix}
        \gamma^3 & \kappa ^4 -1\\ 
        1 & \gamma^3
    \end{pmatrix}
    $$
    which can be obtained by plugging $\phi(x) =x $ and $\phi(x) = x^2-1$ into the operator equation \eqref{Eq:OperatorEquation} above.
    Thus, computing the eigenvalues is a matter of finding the zeroes of the characteristic polynomial. We set:
    $$
    \det(A-cI) = (\gamma^3 - c)^2 - (\kappa^4 -1) = 0 \implies c = \gamma^3 \pm (\kappa^4 - 1)^{1/2} 
    $$
    This gives us the coefficients $b_1, b_2$ as the following - if $\gamma^3 > 0$, then:
    $$
    b_1 = \gamma^3 + (\kappa^{4} - 1)^{1/2}, \quad b_2 = \gamma^{3} - (\kappa^{4}-1)^{1/2}
    $$
    and if $\gamma^3 < 0$ then:
    $$
        b_1 = \gamma^3 - (\kappa^{4} - 1)^{1/2}, \quad b_2 = \gamma^{3} + (\kappa^{4}-1)^{1/2}
    $$
    The reason for this is that we label our eigenvalues in order of absolute value. With the definitions above, we are guaranteed $|b_1| \geq |b_2|$. The number $\kappa^4 - 1$ is always nonnegative
    by Jensen's inequality, but equals zero whenever $\bb{E}[X^4] = \bb{E}[X^2] = 1$. This occurs if and only if $|X|$ is almost surely equal to $1$, which in turn occurs if and only if $X$ is a random variable supported on the two point set $\cbrace{-1,1}$.
    \newline 
    \newline 
    Next, we need to find the corresponding eigenfunctions, so that we can ultimately determine $\mu_k$. 
    Next, we solve for the eigenfunctions. For any coefficients $(p,q)$ we have:
    $$
    \begin{pmatrix}
        \gamma^3 & \kappa ^4 -1\\ 
        1 & \gamma^3
    \end{pmatrix}\begin{pmatrix}
        p \\ q
    \end{pmatrix} = \begin{pmatrix}
        p\gamma^3 + q(\kappa^4 -1) \\ 
        p + q\gamma^3
    \end{pmatrix}
    $$
    Letting $c$ be any eigenvalue, the eigenfunction equations are thus:
    $$
   p\gamma^3 + q(\kappa^4-1)  = cp,\quad  p+ q \gamma^3 = cq
    $$
    We can get one selection by setting $p = 1$. Then, looking at the second equation gives:
    $$
    1 + q\gamma^3 = cq \implies q = \frac{1}{c-\gamma^3}
    $$
    Plugging in the two possible values for $c$ gives::
    $$
    \psi_1(x) = x + \frac{1}{(\kappa^4-1)^{1/2}}(x^2 -1), \quad \psi_2(x) = x - \frac{1}{ (\kappa^{4} -1)^{1/2}}(x^2-1) 
    $$
    Next, we need to normalize these so that they have $L^2$ norm $1$. In general we have:
    $$
    \E[(pX + q(X^2 -1))^2] = \E[p^2 X^2 + 2pq(X^3 - X) + q^2(X^2-1)^2] = p^2 + 2pq\gamma^3 + q^2(\kappa^4-1) 
    $$
    Hence:
    $$
    \E[\psi_1(X)^2] = 1 + \frac{2\gamma^6}{(\kappa^{4}-1)^{1/2}} + \frac{\kappa^4 - 1}{(\kappa^4-1)}  =2 \paren{1 + \frac{2\gamma^{3}}{(\kappa^{4}-1)^{1/2}}}
    $$
    And:
    $$
    \E[\psi_2(X)^2] = 2\paren{1 -\frac{2\gamma^3}{(\kappa^4 -1)^{1/2}}}
    $$
    So that the normalized eigen-functions are:
    $$
    \phi_1(x) = \frac{x + \frac{1}{(\kappa^4-1)^{1/2}}(x^2 -1)}{\paren{2 \paren{1 + \frac{2\gamma^{3}}{(\kappa^{4}-1)^{1/2}}}}^{1/2}}
    $$
    $$\phi_2(x) = \frac{x + \frac{1}{(\kappa^4-1)^{1/2}}(x^2 -1)}{\paren{2 \paren{1 - \frac{2\gamma^{3}}{(\kappa^{4}-1)^{1/2}}}}^{1/2}}$$
    Finally, we need to obtain $a_k$, to represent $H_1(x)$ in terms of the eigen-functions $\phi_1(x), \phi_2(x)$ and the constant eigenfunction $\phi_0(x)$. Notice:
    $$
    H_1(x) = x = \frac{1}{2}\psi_1(x) + \frac{1}{2}\psi_2(x) = \frac{\E[\psi_1(X)^2]^{1/2}}{2}\phi_1(x) + \frac{\E[\psi_2(X)^2]^{1/2}}{2}\phi_2(x)
    $$
    Thus, by definition, we have:
    $$
    a_1 = \frac{\paren{2 \paren{1 + \frac{2\gamma^{3}}{(\kappa^{4}-1)^{1/2}}}}^{1/2}}{2}, \quad a_2 = \frac{\paren{2 \paren{1 - \frac{2\gamma^{3}}{(\kappa^{4}-1)^{1/2}}}}^{1/2}}{2} 
    $$
    \subsection{Student's T-Statistic Kernel - Bound Analysis}
    \label{sec: t stat appendix analysis}
    In this section we provide the computations behind the asymptotics discussed in section \ref{sec: t stat analysis}, in particular in the case when $\gamma > 0$, i.e. $p_1 < 2$. 
    We begin with the bound from Theorem $1$. We first compute some intermediate quantities. In this case, we exactly two eigenvalues and hence:
    $$
    \frac{1}{p_1} + \frac{1}{p_2} = 1 \implies \frac{1}{p_1} - \frac{1}{2} = \frac{1}{2} - \frac{1}{p_2}.
    $$
    We have:
    $$
    \frac{1}{p_1} - \frac{1}{2} = \frac{\gamma^3\xi^{1/2}}{2(\gamma^6 + \xi)}, \quad \frac{p_1}{2} -1 = \frac{\gamma^6 + \xi}{(\gamma^3 + \xi^{1/2})^2} - 1 = \frac{-2\gamma^3\xi}{(\gamma^3 + \xi^{1/2})^2},
    $$
    and
    $$
    \frac{1}{2} - \frac{1}{p_2} = \frac{\gamma^3\xi^{1/2}}{2(\gamma^6 + \xi)}, \quad \frac{p_2}{2} -1 = \frac{\gamma^6 + \xi}{(\gamma^3 - \xi^{1/2})^2} - 1 = \frac{2\gamma^3\xi}{(\gamma^3 + \xi^{1/2})^2}.
    $$
    Using the notation of Theorem 1, we obtain the quantities:
    \begin{align}
    B_{p,1}^{1/p_1} &= \paren{\frac{1}{\frac{1}{p_1} - \frac{1}{2}}}^{1/p_1}\sqbrace{\paren{\frac{\e}{2\L_1}}^{\frac{p_1}{2} - 1}-1}^{\frac{1}{p_1}} \\
    &= \sqbrace{\paren{\frac{2(\gamma^6 + \xi)}{\gamma^3\xi^{1/2}}}\sqbrace{\paren{\frac{\e}{2\L_1}}^{\frac{-2\gamma^3\xi}{(\gamma^3 + \xi^{1/2})^2}} - 1}}^{\frac{1}{2} + \frac{\gamma^3\xi^{1/2}}{2(\gamma^6 + \xi)}} \\
    A_{p,2}^{1/p_2} &= \paren{\frac{1}{\frac{1}{2} - \frac{1}{p_2}}}^{1/p_2} \sqbrace{1 - \paren{\frac{\e}{2\L_1}}^{\frac{p_2}{2} - 1}}^{\frac{1}{p_2}} \\
    &= \sqbrace{\paren{\frac{2(\gamma^6 + \xi)}{\gamma^3\xi^{1/2}}}\sqbrace{1 - \paren{\frac{\e}{2\L_1}}^{\frac{-2\gamma^3\xi}{(\gamma^3 + \xi^{1/2})^2}} }}^{\frac{1}{2} - \frac{\gamma^3\xi^{1/2}}{2(\gamma^6 + \xi)}}
    \end{align}
    Multiplying together and substituting the expression for $\L_1^2$ gives we obtain:
    \begin{align*}
    B_{p,1}^{1/p_1}A_{p,2}^{1/p_2} &= \paren{\frac{2(\gamma^6 + \xi)}{\gamma^{3}\xi^{1/2}}} \sqbrace{\paren{\frac{\e}{2\L_1}}^{{\frac{-2\gamma^3\xi}{(\gamma^3 + \xi^{1/2})^2}}}-1}^{\frac{1}{2} + \frac{\gamma^3\xi^{1/2}}{2(\gamma^6 + \xi)}}
    \sqbrace{1 - \paren{\frac{\e}{2\L_1}}^{\frac{2\gamma^3\xi}{(\gamma^3 + \xi^{1/2})^2}}}^{\frac{1}{2} - \frac{\gamma^3\xi^{1/2}}{2(\gamma^6 + \xi)}} \\
    &=\paren{\frac{2(\gamma^6 + \xi)}{\gamma^{3}\xi^{1/2}}} \sqbrace{\paren{\frac{n\e^2}{4(\gamma^6 + \xi)}}^{{\frac{-\gamma^3\xi}{(\gamma^3 + \xi^{1/2})^2}}}-1}^{\frac{1}{2} + \frac{\gamma^3\xi^{1/2}}{2(\gamma^6 + \xi)}}
    \sqbrace{1 - \paren{\frac{n\e^2}{4(\gamma^6 + \xi)}}^{\frac{\gamma^3\xi}{(\gamma^3 + \xi^{1/2})^2}}}^{\frac{1}{2} - \frac{\gamma^3\xi^{1/2}}{2(\gamma^6 + \xi)}}.
    \end{align*}
    Finally, to obtain the full bound, we must divide by $\L_1$, which in this case is $\frac{[2(\gamma^6 + \xi)]^{1/2}}{n^{1/2}}$. This provides a total bound:
    $$
    \frac{n^{1/2}}{[2(\gamma^6 + \xi)]^{1/2}}\paren{\frac{2(\gamma^6 + \xi)}{\gamma^{3}\xi^{1/2}}} \sqbrace{\paren{\frac{n\e^2}{4(\gamma^6 + \xi)}}^{{\frac{-\gamma^3\xi}{(\gamma^3 + \xi^{1/2})^2}}}-1}^{\frac{1}{2} + \frac{\gamma^3\xi^{1/2}}{2(\gamma^6 + \xi)}}
    \sqbrace{1 - \paren{\frac{n\e^2}{4(\gamma^6 + \xi)}}^{\frac{\gamma^3\xi}{(\gamma^3 + \xi^{1/2})^2}}}^{\frac{1}{2} - \frac{\gamma^3\xi^{1/2}}{2(\gamma^6 + \xi)}}
    $$
    Finally, as is seen in the full bounds of Theorems \ref{Thm:Full Bound}, \ref{Thm:User Friendly Bound}, this term is multiplied by the indicator:
    $$
    \mathbbm{1}\cbrace{\e \leq 2 \L_1} = \mathbbm{1}\cbrace{\e \leq \frac{2(\gamma^6 + \xi)}{n}}.
    $$
    Thus, it only makes sense to analyze the bound in these cases.
    \newline 
    \newline 
    Assume now that:
    $$
    \frac{\gamma^6+\xi}{n} \to c^2,
    $$
    for some $c \in (0, +\infty)$; this is equivalent to saying that the chi-square part of the statistic $n^{1/2}\widetilde{W_n}$ has a non-trivial limit. Then, after fixing $2c >\e>0$ and passing to the limit, the above terms become:
    \begin{align}
    & \quad \frac{1}{c}\paren{\frac{2(\gamma^6 + \xi)}{\gamma^{3}\xi^{1/2}}} \sqbrace{\paren{\frac{\e^2}{4c^2}}^{{\frac{-\gamma^3\xi}{(\gamma^3 + \xi^{1/2})^2}}}-1}^{\frac{1}{2} + \frac{\gamma^3\xi^{1/2}}{2(\gamma^6 + \xi)}}
    \sqbrace{1 - \paren{\frac{\e^2}{4c^2}}^{\frac{\gamma^3\xi}{(\gamma^3 + \xi^{1/2})^2}}}^{\frac{1}{2} - \frac{\gamma^3\xi^{1/2}}{2(\gamma^6 + \xi)}} 
    \\&=\frac{1}{c}\paren{\frac{2(\gamma^6 + \xi)}{\gamma^{3}\xi^{1/2}}} \sqbrace{\paren{\frac{4c^2}{\e^2}}^{{\frac{\gamma^3\xi}{(\gamma^3 + \xi^{1/2})^2}}}-1}^{\frac{1}{2} + \frac{\gamma^3\xi^{1/2}}{2(\gamma^6 + \xi)}}
    \sqbrace{1 - \paren{\frac{\e^2}{4c^2}}^{\frac{\gamma^3\xi}{(\gamma^3 + \xi^{1/2})^2}}}^{\frac{1}{2} - \frac{\gamma^3\xi^{1/2}}{2(\gamma^6 + \xi)}}.
    \end{align}
    Examining this term more closely, we see that if $c$ is large relative to $\e$, the first term in square braces dominates, whereas if $c$ is small relative to $\e$, the coefficient $\frac{1}{c}$ is dominant. Furthermore, as $c \to +\infty$, we observe that the above term tends to $0$. This can be observed more generally by taking $\L_1 \to +\infty$ in the bound for $\mathcal{I}$ in Theorem $\ref{Thm:Full Bound}$. Alternatively, we can analyze the above bound by hand to confirm this behavior - the only term to reckon with is the first in square braces. We have that:
    $$
    \sqbrace{\paren{\frac{4c^2}{\e^2}}^{{\frac{\gamma^3\xi}{(\gamma^3 + \xi^{1/2})^2}}}-1}^{\frac{1}{2} + \frac{\gamma^3\xi^{1/2}}{2(\gamma^6 + \xi)}} =  \sqbrace{\paren{\frac{4c^2}{\e^2}}^{{\frac{\gamma^3\xi}{(\gamma^3 + \xi^{1/2})^2}}}-1}^{\frac{(\gamma^3 + \xi^{1/2})^2}{2(\gamma^6 + \xi)}} \leq \sqbrace{\frac{4c^2}{\e^2}}^{\frac{\gamma^3\xi}{2(\gamma^{6} + \xi)}} .
    $$
   The AM-GM inequality gives that $xy \leq \frac{x^2 + y^2}{2} \implies \frac{\gamma^{3}\xi}{2(\gamma^6+\xi)} \leq \frac{1}{4}$. Hence, by the assumption that $\e \leq 2c$ we obtain:
   $$
   \sqbrace{\frac{4c^2}{\e^2}}^{\frac{\gamma^3\xi}{2(\gamma^{6} + \xi)}}  \leq 4^{1/4} \frac{c^{1/2}}{\e}
   $$
   Multiplying by the leading term $\frac{1}{c}$ shows that this tends to $0$. 
    This term represents the chi-square part of the concentration function. We may also apply the simpler bound of Theorem \ref{Thm:User Friendly Bound}. It is only explicitly dependent on $p_1$ (although in this case, as only $\l_1,\l_2 \neq 0$, $p_1$ determines $p_2$) and gives:
    $$
    \mathcal{J}(\e, \l) \leq \frac{1}{\e}\frac{1}{\frac{2}{p_1} - 1}\paren{\frac{\e}{\L_1}}^{\frac{3}{2} - \frac{1}{p_1}} = \frac{1}{\frac{2}{p_1} - 1} \paren{\frac{1}{\L_1}}^{\frac{3}{2} - \frac{1}{p_1}}\e^{\frac{1}{2} - \frac{1}{p_1}}.
    $$
    We have that:
    \begin{align*}
    \frac{3}{2} - \frac{1}{p_1} &= \frac{1}{2} + \frac{1}{p_1} + \frac{1}{p_2} -\frac{1}{p_1}  \\
    &= \frac{1}{2} + \frac{1}{p_2} \\
    &= \frac{1}{2} + \frac{(\gamma^3 - \xi^{1/2})^2}{2(\gamma^6 + \xi)}.
    \end{align*}
    and:
    $$
    \frac{2}{p_1} - 1 = 2\paren{\frac{1}{p_1} - \frac{1}{2}} = \frac{\gamma^3\xi^{1/2}}{(\gamma^6 + \xi)}.
    $$
    Thus, in this case, we obtain an estimate for $\e \mathcal{J}(\e,\l)$ of the form:
    $$
    \frac{(\gamma^6 + \xi)}{\gamma^3\xi^{1/2}}\paren{\frac{n^{1/2}}{2^{3/2}(\gamma^6 + \xi)^{1/2}}}^{\frac{3}{2} - \frac{\paren{\gamma^3 - \xi^{1/2}}^2}{2(\gamma^6 + \xi)}} \e^{-\frac{\gamma^3\xi^{1/2}}{2(\gamma^6 + \xi)}}.
    $$
    Taking $n \to \infty$ yields:
    $$
     \frac{(\gamma^6 + \xi)}{\gamma^3\xi^{1/2}}\paren{\frac{1}{2^{3/2}c}}^{\frac{3}{2} - \frac{\paren{\gamma^3 - \xi^{1/2}}^2}{2(\gamma^6 + \xi)}} \e^{-\frac{\gamma^3\xi^{1/2}}{2(\gamma^6 + \xi)}}.
    $$
    To recover the full bound, we multiply by the appropriate indicator as above. 
\end{document}